\documentclass[reqno, 12pt]{amsart}
\usepackage[mathscr]{eucal}
\usepackage{amscd}
\usepackage{color}
\usepackage{amsfonts}
\usepackage{amsmath, amssymb}
\usepackage{latexsym}
\usepackage{psfrag}
\usepackage{xpatch}
\usepackage{graphicx}
\usepackage{subfigure}

\usepackage{graphicx}

\numberwithin{equation}{section}
\newcommand{\de}{\delta}
\newcommand{\si}{\sigma}

\newcommand{\ga}{\gamma}

\newcommand{\Om}{\Omega}
\newcommand{\la}{\lambda}
\newcommand{\ol}{\overline}

\newcommand{\no}{\nonumber}
\newcommand{\pl}{\partial}
\newcommand{\Si}{\Sigma}
\newcommand{\vep}{\varepsilon}

\newcommand{\al}{\alpha}
\newcommand{\avint}{\rlap{$-$}\!\int_{\Omega_0}}

\newcommand{\R}{\mathbb{R}}

\newtheorem{theorem}{Theorem}[section]
\newtheorem{lemma}{Lemma}[section]
\newtheorem{proposition}{Proposition}[section]

\newtheorem{remark}{Remark}[section]

\def\be{\begin{equation}}
\def\ee{\end{equation}}
\def\bge{\begin{eqnarray}}
\def\bgee{\begin{eqnarray*}}
\def\ege{\end{eqnarray}}
\def\egee{\end{eqnarray*}}
\begin{document}

\title{Dynamics of shadow system of a singular Gierer-Meinhardt system on an evolving domain}
\author{Nikos I. Kavallaris}
\address{
 Department of Mathematics, University of Chester, Thornton Science Park
Pool Lane, Ince, Chester  CH2 4NU, UK
}
\email{n.kavallaris@chester.ac.uk}

\author{Raquel Bareira}
\address{Barreiro School of Technology of the Polytechnic Institute of Setubal\\
Rua Americo da Silva Marinho-Lavradio, 2839-001 Barreiro, Portugal\\
and
CMAFcIO - Center of Mathematics, Fundamental Applications and Operations Research, University of Lisbon, Portugal }
\email{raquel.barreira@estbarreiro.ips.pt}

\author{Anotida Madzvamuse}
\address{School of Mathematical and Physical Sciences\\
Department of Mathematics\\
University of Sussex\\
Falmer, Brighton, BN1 9QH, England, UK}

\email{a.madzvamuse@sussx.ac.uk}

\subjclass{Primary: 35B44, 35K51 ; Secondary: 35B36, 92Bxx }

\keywords{Pattern formation, Turing instability, activator-inhibitor system, shadow-system, invariant regions, diffusion-driven blow-up, evolving domains}
\date{\today}
\maketitle
\begin{abstract}
 The main purpose  of the current paper is to contribute towards the comprehension of the dynamics of the shadow system of  a singular Gierer-Meinhardt model on an isotropically evolving domain.
In the case where the inhibitor's response to the activator's growth is rather weak, then the shadow system of the Gierer-Meinhardt model is reduced to a single though non-local equation whose dynamics is thoroughly investigated throughout the manuscript. The main focus is on the derivation of blow-up results for this non-local equation, which can be interpreted  as instability patterns of the shadow system. In particular, a {\it diffusion-driven  instability (DDI)}, or {\it Turing instability},  in the neighbourhood of a constant stationary solution, which then is destabilised  via diffusion-driven blow-up, is observed. The latter indicates the formation of some unstable patterns, whilst some stability results of global-in-time solutions towards non-constant steady states guarantee the occurrence of some stable patterns. Most of the derived results are confirmed numerically and also compared with the ones in the case of a stationary domain.
\end{abstract}

\subjclass{Primary: 35B44, 35K51 ; Secondary: 35B36, 92Bxx }

\keywords{Pattern formation, Turing instability, activator-inhibitor system, shadow-system, invariant regions, diffusion-driven blow-up, evolving domains}

\section{Introduction}
The purpose of the current work is to study an activator-inhibitor system, introduced by Gierer and Meinhard in 1972 \cite{gm72} to describe the phenomenon of morphogenesis in hydra, on an evolving domain. Assume that $u(x,t)$ stands for the concentration of the activator, at a spatial point $x\in \Omega_t\subset \R^N, N=1,2,3,$ at time $t\in[0,T], T>0,$ which enhances its own production and that of the inhibitor. On the other hand, let $v(x,t)$ represents the concentration of the inhibitor, which suppresses its own production as well as that of the activator. Hence, the interaction between $u$ and $v$ can be described by the following non-dimensionalised system \cite{gm72}
\bge
&&u_t+\nabla \cdot (\overrightarrow{\alpha} u)= D_1 \Delta u-u+\displaystyle\frac{u^p}{v^q}, \quad x\in \Omega_t,\; t\in (0,T), \label{egm1}\\
&&\tau v_t+\nabla \cdot (\overrightarrow{\alpha} v) = D_2 \Delta v-v+\displaystyle\frac{u^r}{v^s}, \quad \; x\in \Omega_t,\; t\in (0,T),
 \label{egm2}\\
 && \frac{\pl u}{\pl \nu}=\frac{\pl v}{\pl \nu}=0 \quad \; x\in \pl\Omega_t,\; t\in (0,T),\label{egm3}\\
 && u(x,0)=u_0(x)>0,\quad v(x,0)=v_0(x)>0, \quad x\in \Om_0\subset \R^N, \label{egm4}
\ege
where $\nu$ is the unit normal vector on $\pl \Om_t$, whereas $\overrightarrow{\alpha}\in \R^N$ stands for the convection velocity which is induced by the material deformation due to the evolution of the domain. Moreover, $D_1,D_2$ are the diffusion coefficients of the activator and inhibitor respectively; $\tau$ represents the response of the inhibitor to the activator's growth. Moreover, the exponents satisfying the conditions:
$p>1,\; q, r,>0,\;\mbox{and}\; s>-1,$ measure the interactions between morphogens. The dynamics of system \eqref{egm1}-\eqref{egm4}  can be characterised by two values: the net self-activation index $\pi=( p-1)/r$ and the net cross-inhibition
index $\gamma=q/(s+1).$ Index $\pi$ correlates the strength of self-activation of the activator
with the cross-activation of the inhibitor. Thus, if $\pi$ is large, then the net growth of the
activator is large no matter the growth of the inhibitor. The  parameter $\gamma$ measures how strongly
the inhibitor suppresses the production of the activator and that of itself. If $\gamma$ is large then
the production of the activator is strongly suppressed by the inhibitor. Finally, the parameter $\tau$
quantifies the inhibitor's response against the activator's growth.
Guided by biological interpretation as well as by mathematical reasons, we assume that the parameters $p, q, r, s $ satisfy the condition
\bge\label{tc}
p-r\gamma<1,
\ege
which in the literature is known as the {\it Turing condition} since it guarantees the occurrence of Turing patterns for the system \eqref{egm1}-\eqref{egm4} on a stationary domain \cite{nst06}.

For analytical purposes, in the current work we will only consider the case of an isotropic flow on an evolving domain, and thus  we have for any $x\in \Om_t$:
\bge
x=\rho(t) \xi,\quad\mbox{for}\quad \xi\in \Om_0\subset \R^N,\label{isf}
\ege
with $\rho(t)$ being  $C^1-$function with $\rho(0)=1.$ In the case of a growing domain we have $\dot{\rho}(t)=\frac{d \rho}{dt}>0,$ whilst when the domain shrinks or for domain contraction  $\dot{\rho}(t)=\frac{d \rho}{dt}<0.$  Furthermore,  the following equality holds
\bge
\frac{dx}{dt}=\overrightarrow{\alpha}(x,t).\label{cisf}
\ege
Setting
$
\hat{u}(\xi,t)=u(\rho(t)\xi,t),\; \hat{v}(\xi,t)=v(\rho(t)\xi,t),
$
and then using the chain rule as well as \eqref{isf} and \eqref{cisf}, see also \cite{mm07}, we obtain:
\bgee
&&\hat{u}_t-\overrightarrow{\alpha}\cdot \nabla_x u=u_t,\quad \nabla_x u=\frac{1}{\rho(t)}\nabla_{\xi}\hat{u}\\
&&\Delta_x u= \frac{1}{\rho^2(t)}\Delta_{\xi}\hat{u},\quad \nabla_x \cdot \left(\overrightarrow{\alpha} u\right)=\overrightarrow{\alpha}\cdot \nabla_x u+N\,u \frac{\dot{\rho}(t)}{\rho(t)},
\egee
whilst similar relations hold for $v$ as well. Therefore \eqref{egm1}-\eqref{egm4} is reduced to following system on a reference stationary domain $\Om_0$
\bge
&&\hat{u}_t= \frac{D_1}{\rho^2(t)} \Delta_{\xi} \hat{u}-\left(1+N\frac{\dot{\rho}(t)}{\rho(t)}\right)\hat{u}+\displaystyle\frac{\hat{u}^p}{\hat{v}^q}, \quad \xi\in \Omega_0,\; t\in (0,T), \label{regm1}\\
&&\tau \hat{v}_t=\frac{D_2}{\rho^2(t)} \Delta_{\xi} \hat{v}-\left(1+N\frac{\dot{\rho}(t)}{\rho(t)}\right)\hat{v}+\displaystyle\frac{\hat{u}^r}{\hat{v}^s}, \quad \; \xi\in \Omega_0,\; t\in (0,T),
 \label{regm2}\\
 && \frac{\pl \hat{u}}{\pl \nu}=\frac{\pl \hat{v}}{\pl \nu}=0 \quad \; \xi\in \pl\Omega_0,\; t\in (0,T),\label{regm3}\\
 && \hat{u}(\xi,0)=\hat{u}_0(\xi)>0,\quad \hat{v}(\xi,0)=\hat{v}_0(\xi)>0, \quad \xi\in \Om_0, \label{regm4}
\ege
where $\Delta_{\xi}$ represents the Laplacian on the reference static domain $\Om_0.$
Henceforth, without any loss of generality we will omit the index $\xi$ from the Laplacian.

Defining a new time scale \cite{L11},
\bge\label{aal2}
\si(t)=\int_0^t\frac{1}{\rho^2(\theta)}\,d\theta,
\ege
and setting $\tilde{u}(\xi,\sigma)=\hat{u}(\xi, t), \tilde{v}(\xi,\sigma)=\hat{v}(\xi, t),$ then system \eqref{regm1}-\eqref{regm4} can be written as
\bge
&&\tilde{u}_{\sigma}= D_1 \Delta_{\xi} \tilde{u}-\left(\phi^2(\si)+N\frac{\dot{\phi}(\si)}{\phi(\si)}\right)\tilde{u}+\phi^2(\si)\displaystyle\frac{\tilde{u}^p}{\tilde{v}^q}, \quad \xi\in \Omega_0,\; \si\in (0,\Sigma), \label{tregm1}\\
&&\tau \tilde{v}_{\si}=D_2 \Delta_{\xi}  \tilde{v}-\left(\phi^2(\si)+N\frac{\dot{\phi}(\si)}{\phi(\si)}\right)\tilde{v}+\phi^2(\si)\displaystyle\frac{\tilde{u}^r}{\tilde{v}^s}, \quad \; \xi\in \Omega_0,\; \si\in (0,\Sigma),
 \label{tregm2}\\
 && \frac{\pl \tilde{u}}{\pl \nu}=\frac{\pl \tilde{v}}{\pl \nu}=0, \quad \; \xi\in \pl\Omega_0,\; \si\in (0,\Sigma),\label{tregm3}\\
 && \tilde{u}(\xi,0)=\hat{u}_0(\xi)>0,\quad \tilde{v}(\xi,0)=\hat{v}_0(\xi)>0, \quad \xi\in \Om_0, \label{tregm4}
\ege
where $ \rho(t)=\phi(\si),$ and thus $\dot{\rho}(t)=\frac{\dot{\phi}(\si)}{\phi^2(\si)},$ and $\Sigma=\sigma(T).$

Now if $D_1\ll D_2,$ i.e. when the inhibitor diffuses much faster than the activator, then system \eqref{tregm1}-\eqref{tregm4} can be fairly approximated by an ODE-PDE system with a non-local reaction term. We will denote the new approximation by {\it shadow system} as coined in \cite{ke78}. Below we provide a rather rough derivation of the shadow system, while for a more rigorous approach one can appeal to the arguments in \cite{bk18}. Indeed, dividing \eqref{tregm2} by $D_2$ and taking $D_2\to +\infty,$ see also \cite{n11},  then it follows that $\tilde{v}$ solves
$
\Delta_{\xi}  \tilde{v}=0, \quad \; \xi\in \Omega_0,\quad
\frac{\pl \tilde{v}}{\pl \nu}=0, \quad \; \xi\in \pl\Omega_0,
$
for any fixed $\si\in(0,\Sigma).$ Due to the imposed Neumann boundary condition then $\tilde{v}$ is a spatial homogeneous (independent of $\xi$) solution, i.e. $\tilde{v}(\xi,\si)=\eta(\si)$ and thus \eqref{tregm2} can be written as
\bge\label{rq3}
\tau \frac{d\eta}{d \si}=-\Phi(\si)\eta+\phi^2(\si)\displaystyle\frac{\tilde{u}^r}{\eta^s}, \quad \si\in (0,\Sigma),
\ege
where
\bge\label{ps1}
\Phi(\si)=:\left(\phi^2(\si)+N\frac{\dot{\phi}(\si)}{\phi(\si)}\right).
\ege
 Averaging \eqref{rq3} over $\Om_0$ we finally infer that the pair $(\tilde{u}, \eta)$ satisfies the {\it shadow system}\rm
\bge
&&\tilde{u}_{\sigma}= D_1 \Delta_{\xi}  \tilde{u}-\Phi(\si)\tilde{u}+\phi^2(\si)\displaystyle\frac{\tilde{u}^p}{\eta^q}, \quad \xi\in \Omega_0,\; \si\in (0,\Sigma), \label{stregm1}\\
&&\tau \frac{d\eta}{d\si}=-\Phi(\si)\eta+\phi^2(\si)\displaystyle\frac{\avint\tilde{u}^r\,d\xi}{\eta^s}, \quad  \si\in (0,\Sigma),
 \label{stregm2}\\
 && \frac{\pl \tilde{u}}{\pl \nu}=0, \quad \; \xi\in \pl\Omega_0,\; \si\in (0,\Sigma),\label{stregm3}\\
 && \tilde{u}(\xi,0)=\hat{u}_0(\xi)>0,\quad \eta(0)=\eta_0>0, \quad \xi\in \Om_0, \label{stregm4}
\ege
where
$
\avint\tilde{u}^r\,d\xi=\frac{1}{|\Om_0|}\int_{\Om_0} \tilde{u}^r\,d\xi.
$
In the limit case $\tau\to 0,$ i.e. when the inhibitor's response to the growth of the activator is quite small, then the shadow system is reduced to a single, though, non-local equation. Indeed, when $\tau=0$, \eqref{stregm2} entails that
$
\eta(\si)=\left(\frac{\phi^2(\si)}{\Phi(\si)}\avint\tilde{u}^r\,d\xi\right)^{\frac{1}{s+1}},
$
and thus \eqref{stregm1}-\eqref{stregm4} reduce to
\bge
&&\tilde{u}_{\sigma}= D_1 \Delta_{\xi}  \tilde{u}-\Phi(\si)\tilde{u}+\displaystyle\frac{\Psi(\si)\tilde{u}^p}{\left(\avint\tilde{u}^r\,d\xi\right)^{\ga}}, \quad \xi\in \Omega_0,\; \si\in (0,\Sigma), \label{nstregm1}\\
&& \frac{\pl \tilde{u}}{\pl \nu}=0, \quad \; \xi\in \pl\Omega_0,\; \si\in (0,\Sigma),\label{nstregm2}\\
&&\tilde{u}(\xi,0)=\hat{u}_0(\xi)>0,\quad \xi\in \Om_0, \label{nstregm3}
\ege
recalling $\ga=\frac{q}{s+1}$ and
\bge\label{ps2}
\Psi(\si)=\phi^{2(1-\ga)}(\si)\Phi^{\ga}(\si).
\ege
Recovering the $t$ variable entails that the following partial differential equation holds
\bge
&&\hat{u}_{t}= \frac{D_1}{\rho^2(t)} \Delta_{\xi}  \hat{u}-L(t)\hat{u}+L^{-\ga}(t)\displaystyle\frac{\hat{u}^p}{\left(\avint\hat{u}^r\,d\xi\right)^{\ga}}, \quad \xi\in \Omega_0,\; t\in (0,T), \label{nstregm1t}\\
&& \frac{\pl \hat{u}}{\pl \nu}=0, \quad \; \xi\in \pl\Omega_0,\; t\in (0,T),\label{nstregm2t}\\
&&\hat{u}(\xi,0)=\hat{u}_0(\xi)>0,\quad \xi\in \Om_0, \label{nstregm3t}
\ege
where $L(t):=\left(1+N\frac{\dot{\rho}(t)}{\rho(t)}\right).$ We note that formulation \eqref{nstregm1}-\eqref{nstregm3} is more appropriate for the demonstrated mathematical analysis, however all of our theoretical results can be directly interpreted in terms of the equivalent formulation
\eqref{nstregm1t}-\eqref{nstregm3t}. Besides,  formulation
\eqref{nstregm1t}-\eqref{nstregm3t} is more appropriate for our numerical experiments since the calculation of the functions $\Phi(\sigma)$ and $\Psi(\sigma)$ is not always possible.

The main aim of the current work is to investigate the long-time dynamics of the non-local problem \eqref{nstregm1}-\eqref{nstregm3} and then check whether it resembles that of the reaction-diffusion system \eqref{tregm1}-\eqref{tregm4}. Biologically speaking we will investigate whether, under the fact that the inhibitor's response to the growth of the activator is quite small and when it also diffuses much faster than the activator,  is it necessary to study the dynamics of both reactants or  it is sufficient to study only the activator's dynamics. From here onwards, we take $D_1=1$, revert to the initial variables $x, u$  instead of $\xi, \widetilde{u}$  and we drop the index $\xi$ from the Laplacian $\Delta$ without any loss of generality. Hence, we will focus our study on the following single partial differential equation
\bge
&&u_{\si}= \Delta u-\Phi(\si)u+\displaystyle\frac{\Psi(\si)u^p}{\left(\avint u^r\,dx\right)^{\ga}}, \quad x\in \Omega_0,\; \si\in (0,\Sigma), \label{nstregm1k}\\
&& \frac{\pl u}{\pl \nu}=0, \quad \; x\in \pl\Omega_0,\; \si\in (0,\Sigma),\label{nstregm2k}\\
&&u(x,0)=u_0(x)>0,\quad x\in \Om_0. \label{nstregm3k}
\ege
The layout of the current work is as follows. Section \ref{ode-bge} deals with the derivation and proofs of various blow-up results, induced by the non-local reaction term (ODE blow-up results), together with some global-time existence results  for problem \eqref{nstregm1k}-\eqref{nstregm3}. Following the approach developed in \cite{KS16,KS18}, in Section \ref{tipf} we present and prove a Turing instability result associated with \eqref{nstregm1k}-\eqref{nstregm3}. This Turing instability occurs under the {\it Turing condition} \eqref{tc}  and is exhibited in the form of a {\it driven-diffusion} finite-time blow-up. Finally, in Section \ref{num-sec} we appeal to various numerical experiments in order to confirm some of the theoretical results presented in Sections \ref{ode-bge} and \ref{tipf}. We also compare numerically the long-time dynamics of the non-local problem \eqref{nstregm1k}-\eqref{nstregm3} with that of the reaction-diffusion system \eqref{stregm1}-\eqref{stregm4}.

\section{ODE Blow-up and Global Existence}\label{ode-bge}
The current section is devoted to the presentation of some blow-up results for problem \eqref{nstregm1k}-\eqref{nstregm3k}, i.e. blow-up results induced by the kinetic (non-local) term in \eqref{nstregm1k}. Besides, some global-in-time existence results for problem \eqref{nstregm1k}-\eqref{nstregm3k} are also presented. Throughout the manuscript we use the notation $C$ and $c$ to denote positive constants with big and small values respectively. Our first observation is that the concentration of the activator cannot extinct in finite time. Indeed, the following proposition holds.
\begin{proposition}\label{lbd}
Assume that
\bge\label{mts2an}
\inf_{(0,\Sigma)} \Psi(\si):=m_{\Psi}>0,\; \inf_{(0,\Sigma)} \Phi(\si):=m_{\Phi}>0\;\mbox{and}\; \sup_{(0,\Sigma)} \Phi(\si):=M_{\Phi}<+\infty\quad,
\ege
 then for each $\Sigma>0$ there exists $C_{\Sigma}>0$ such that for the solution $u(x,\si)$ of \eqref{nstregm1k}-\eqref{nstregm3k} the following inequality holds
\begin{equation} \label{jg0}
u(x,\si)\geq C_{\Sigma}\quad \mbox{in}\quad \Om_0\times [0,\Sigma).
\end{equation}
\end{proposition}
\begin{proof}
Owing to the maximum principle and by using \eqref{mts2an} we derive that $u=u(x,\si)>0.$ By virtue of the comparison principle,  we also deduce that $u(x,\si)\geq \tilde u(\si)$, where $\tilde{u}=\tilde{u}(\si)$ is the solution to
$ \frac{d\tilde{u}}{d\si}=-M_{\Phi}\tilde{u}\quad \mbox{in $(0, \Sigma)$},\quad \tilde{u}(0)=\tilde{u}_0\equiv \inf_{\Om_0} u_0(x)>0, $
and thus \eqref{jg0} is satisfied with $C_{}=\tilde{u}_0e^{-M_{\Phi} \Sigma}$.
\end{proof}
\begin{remark}\label{rem1}
It is easily checked that condition \eqref{mts2an} is satisfied for any decreasing function $\phi(\sigma)$ satisfying
\bge\label{diq1}
\phi(\si)> \frac{1}{\sqrt{2N\sigma+1}},\; 0<\si<\Sigma,
\ege
since then  by virtue of \eqref{ps1}
\bge\label{ts1}
0<\Phi(\si)=\left(\phi^2(\si)+N\frac{\dot{\phi}(\si)}{\phi(\si)}\right)< \phi^2(\si)< \phi^2(0)=1,\quad 0<\sigma<\Sigma.
\ege
Then \eqref{ts1} via \eqref{ps2} implies that
\bge\label{ts2a}
0<\Psi(\si)=\left(\phi(\si)\right)^{2(1-\gamma)} \Phi^{\gamma}(\si)<1 ,\quad\mbox{for}\quad 0<\gamma<1,\quad 0<\sigma<\Sigma
\ege
and
\bge\label{ts2c}
0<\Psi(\si)=\left(\phi(\si)\right)^{2(1-\gamma)} \Phi^{\gamma}(\si)<m^{2(1-\gamma)}_{\Phi} ,\quad\mbox{for}\quad \gamma>1,\quad 0<\sigma<\Sigma,
\ege
when $m_{\Phi}=\inf_{(0,\Sigma)}\Phi(\si)>0.$
\end{remark}
 A key estimate for obtaining some blow-up results presented throughout is the following proposition.
\begin{proposition}\label{eiq}
Let $\Psi(\si)$ and $\Phi(\si)$ satisfy \eqref{mts2an}, then there exists $\delta_0>0$ such for any $0<\delta\leq \delta_0$ the following estimate  is fulfilled
\begin{equation}
\avint u^{-\delta}\leq C\quad\mbox{for any}\quad 0<\si<\Sigma,
 \label{eqn:1.16h}
\end{equation}
where the constant $C$ is independent of time $\si.$
 \label{thm:1.1}
\end{proposition}
\begin{proof}
Define $\chi=u^{\frac{1}{\al}}$ for $\al\neq 0$, then we can easily check that $\chi$ satisfies
\bge
&&\al \chi_{\si}=\alpha\left(\Delta \chi+4 (\al-1) \vert\nabla \chi^{\frac{1}{2}}\vert^2\right)-\Phi\chi+\frac{\Psi u^{p-1+\frac{1}{\al}}}{\left(\avint u^r\right)^{\ga}} \quad \mbox{in}\quad \Om_0\times (0,\Sigma),\label{ob1}\\
&& \frac{\pl \chi}{\pl \nu}=0, \quad \mbox{on}\quad \partial\Om_0\times (0,\Sigma),\label{ob2a}\\
&&\chi(x,0)=u_0^{\frac{1}{\al}}(x),\quad\mbox{in}\quad \Omega_0.\label{ob2}
\ege
Averaging (\ref{ob1}) over $\Om_0$, we obtain
\begin{equation}\label{ob4}
\alpha\frac{d}{d\si}\avint \chi+4\alpha(1-\alpha)\avint\vert \nabla \chi^{\frac{1}{2}}\vert^2+\Phi\avint \chi= \frac{\avint\Psi u^{p-1+\frac{1}{\al}}}{\left(\avint u^r\right)^{\ga}},
\end{equation}
and hence
\bge\label{ts2}
\frac{d}{d\si}\avint \chi+4(1-\alpha)\avint\vert\nabla \chi^{\frac{1}{2}}\vert^2+\frac{\Phi}{\alpha}\avint \chi\leq 0,
\ege
for $\alpha<0.$ Setting  $\delta=-\frac{1}{\alpha}$ we have
\[ \frac{d}{d\si}\avint \chi+4(1+\delta^{-1})\avint\vert \nabla \chi^{\frac{1}{2}}\vert^2 \leq M_{\Phi}\delta \avint \chi. \]
Now, recall that Poincar\'e-Wirtinger's inequality, \cite{br11}, reads
\bge\label{pwi}
\Vert \nabla w\Vert_2^2\geq \mu_2\Vert w\Vert_2^2,\quad\mbox{for any}\quad w\in H^1(\Om),
\ege
where $\mu_2$ is the second eigenvalue of the Laplace operator associated with Neumann boundary conditions. Then \eqref{ts2} by virtue of \eqref{pwi} for $w=\chi^{\frac{1}{2}}$ entails that
$\frac{d}{d\si}\avint \chi+c\avint \chi\leq 0$,
for some positive constant $c,$ provided $0<\delta\ll 1$. Consequently,  Gr\"{o}wnwall's lemma  yields that $\chi(\si)\leq C<\infty$ for any $0<\si<\Sigma$ and thus (\ref{eqn:1.16h}) follows due to the fact that $\chi=u^{-\delta}.$
\end{proof}
\begin{remark}\label{ts3}
Note that Proposition \ref{thm:1.1}  guarantees that the  non-local term of problem \eqref{nstregm1k}-\eqref{nstregm3k} stays away from zero and hence its solution $u$ is bounded away from zero as well.
In fact, inequality (\ref{eqn:1.16h}) implies $\displaystyle{ \avint u^{\delta}\geq c=C^{-1}}$ and then
\begin{equation}
\avint u^r\geq \left(\avint u^\delta \right)^{r/\delta}\geq c^{r/\delta}>0\quad\mbox{for any}\quad 0<\si<\Sigma,
 \label{eqn:1.17}
\end{equation}
follows by Jensen's inequality, \cite{ev10}, taking $\delta\leq r,$ where again $c$ is independent of time $t.$  The latter estimate rules out the possibility of (finite or infinite time) quenching, i.e.
$
\lim_{\si\to \Sigma} ||u(\cdot,\si)||_{\infty}=0,
$
cannot happen, and thus extinction of the activator in the long run is not possible.
\end{remark}
\begin{remark}\label{nny}
In case $\Phi(\si)$ is not bounded from above, as it happens for $\rho(t)=e^{\beta t},\beta>0,$ when
$\Phi(\si)=(1+N\beta)(1-2\beta \si)^{-1}, 0<\si<\frac{1}{2\beta},$ then both of the estimates \eqref{eqn:1.16h} and \eqref{eqn:1.17} still hold true, however the involved constants depend on time $\si$ and thus  (finite or infinite time) quenching cannot be ruled out.
\end{remark}

Next we present our first ODE-type blow-up result for problem \eqref{nstregm1k}-\eqref{nstregm3k} when an {\it anti-Turing condition}, the reverse of \eqref{tc}, is satisfied.
\begin{theorem}\label{thm1}
Take $p \geq r, 0<\ga<1$ and  $\omega=p-r\gamma>1.$ Assume also $\Psi(\si)>0$ and consider initial data $u_0(x)$ such that
\bge\label{bid}
\bar{u}_0:=\avint u_0\,dx>(\omega-1)^{\frac{1}{1-\omega}}\,I^{\frac{1}{1-\omega}}(\Sigma)>0,
\ege
provided that
\bge\label{nk2}
I(\Sigma):=\int_0^{\Sigma} \Psi(\theta) e^{(1-\omega)\int^{\theta} \Phi(\eta)\,d\eta}\,d \theta<\infty,
\ege
then the solution of \eqref{nstregm1k}-\eqref{nstregm3k} blows up in  finite time $\Sigma_b<\Sigma$, \\i.e.
$\lim_{\si\to \Sigma_b}\Vert u(\cdot,\sigma)\Vert_\infty=+\infty.$
 \label{lem:1.2}
\end{theorem}
\begin{proof}
Since $p>1$ and $p\geq r$, then by virtue of the H\"{o}lder's inequality
$ \avint u^p\geq \left( \avint u\right)^{p}$ and $ \left( \avint u^r\right)^\gamma\leq \left( \avint u^p \right)^{\frac{\gamma r}{p}}.$
Then $\bar{u}(\si)=\avint u(x,\si)\,dx$ satisfies
\begin{equation}
\displaystyle\frac{d \bar{u}}{d\si}= -\Phi(\si)\bar{u}+\Psi(\si)\displaystyle{\frac{\avint u^p}{\left(\avint u^r\right)^{\gamma}}}\geq -\Phi(\si)\bar{u}+\Psi(\si)\bar{u}^{p-r\ga}\quad\mbox{for}\quad 0<\si<\Sigma.
 \label{eqn:1.14}
\end{equation}
Set now $F(\si)$ to be the solution of the following Bernoulli's type initial value problem
$
\frac{d F}{d\si}=-\Phi(\si)F(\si)+\Psi(\si)F^{\omega}(\si),\; 0<\si<\Sigma,\quad F(0)=\bar{u}_0>0,
$
then via the comparison principle $F(\si)\leq \bar{u}(\si)$ for $0<\si<\Sigma$ and $F(\si)$ is given by
$
F(\si)=e^{(\omega-1)\int^{\si} \Phi(\eta)\,d\eta}(G(\si))^{\frac{1}{1-\omega}},
$
where
$
G(\si):=\left[\bar{u}_0^{1-\omega}-(\omega-1)\int_0^{\si} \Psi(\theta)e^{(1-\omega)\int^{\theta} \Phi(\eta)\,d\eta}\,d \theta\right].
$
Note that $F(\si)$ blows up in finite-time if there exists $\si^*<\Sigma$ such that $G(\si^*)=0.$
First note that $G(0)>0;$ furthermore, under the assumption \eqref{bid} we have $\lim_{\si\to \Sigma}G(\si)<0$ and thus by virtue of the intermediate value theorem there exists $\si^*<\Sigma$ such that $G(\si^*)=0.$ The latter implies that $\lim_{\si\to \si^*} F(\si)=+\infty$ and
therefore $\lim_{s\to \Sigma_b} \bar{u}(\si)=+\infty$ for some $\Sigma_b\leq\si^*,$ which completes the proof.
\end{proof}

\begin{remark}\label{aal5}
Note that for an exponentially growing domain, i.e. when $\rho(t)=e^{\beta t}, \beta>0,$ condition \eqref{nk2} is satisfied since then
$1<\Phi(\si)=\left(1+N\beta\right)(1-2\beta \sigma)^{-1}$ and
$1<\Psi(\si)=\left(1+N\beta\right)^{\ga}(1-2\beta \sigma)^{-1}$
for all $\si \in\left(0,\frac{1}{2\beta}\right)$.
Thus
\bgee
 I(\Sigma)=\left(1+N\beta\right)^{\ga}\int_0^{\frac{1}{2 \beta}} \left(1-2\beta \theta\right)^{\frac{(\omega-1)(1+N\beta)}{2\beta}-1}\,d \theta=\frac{\left(1+N\beta\right)^{\ga-1}}{(\omega-1)}<+\infty,
\egee
and according to Theorem \ref{thm1} finite-time blow-up takes place at time
\bge\label{am1}
\Sigma_b\leq\si_{g}=\frac{1}{2\beta}\left\{1-\left[1-(1+N\beta)^{1-\ga}\bar{u}_0^{1-\omega}\right]^{\frac{2 \beta}{(\omega-1)(1+N\beta)}}\right\},
\ege
and for initial data $u_0$ satisfying
$
\bar{u}_0>\left(1+N\beta\right)^{\frac{1-\ga}{\omega-1}}.
$
Besides, for an exponential shrinking domain, i.e. when $\rho(t)=e^{-\beta t}, 0<\beta<\frac{1}{N},$ then again condition \eqref{nk2} is valid since then
\bge\label{aal3}
0<\Phi(\si)=\left(1-N\beta\right)(1+2\beta \sigma)^{-1}<1,\quad \si\in(0,\infty),
\ege and
\bge\label{aal4}
0<\Psi(\si)=\left(1-N\beta\right)^{\ga}(1+2\beta \sigma)^{-1}<1,\quad \si\in(0,\infty).
\ege
In that case
\bgee
 I(\Sigma)=\left(1-N\beta\right)^{\ga}\int_0^{+\infty} \left(1+2\beta \theta\right)^{\frac{(\omega-1)(1-N\beta)}{2\beta}-1}\,d \theta=\frac{\left(1-N\beta\right)^{\ga-1}}{(\omega-1)}<+\infty,
\egee
and again finite-time blow-up occurs at
\bge\label{am2}
\Sigma_b\leq \si_{s}=\frac{1}{2\beta}\left\{\left[1-(1-N\beta)^{1-\ga}\bar{u}_0^{1-\omega}\right]^{\frac{2 \beta}{(1-\omega)(1-N\beta)}}-1\right\},
\ege
provided that the initial data satisfy
$
\bar{u}_0>\left(1-N\beta\right)^{\frac{1-\ga}{\omega-1}}.
$
For a stationary domain, i.e. when $\rho(t)=\phi(\si)=1,$ we have $\Phi(\si)=\Psi(\si)=1$ and thus finite-time blow-up occurs at $\Sigma_1\leq\si_1,$ provided that $\bar{u}_0>1,$ see also \cite{KS16, KS18}, where
$
\si_1:=\frac{1}{1-\omega}\ln\left(1-\bar{u}_0^{1-\omega}\right).
$
\end{remark}
\begin{remark}
When the domain evolves logistically, a feasible choice in the context of biology \cite{pspbm04},  i.e. when
$
\rho(t)=\frac{e^{\beta t}}{1+\frac{1}{m}\left(e^{\beta t}-1\right)},\quad\mbox{for}\quad m\neq1,
$
means that \eqref{aal2} cannot be solved for $t$ and it is more convenient to deal with problem \eqref{nstregm1t}-\eqref{nstregm3t} instead. Then following the same approach as in Theorem \ref{thm1} we show that the solution of \eqref{nstregm1t}-\eqref{nstregm3t} exhibits finite-time blow-up under the same conditions for parameters $p,\gamma,r$ provided that the initial condition satisfies
\bge\label{sk1}
\bar{u}_0:=\avint u_0\,dx>(\omega-1)^{\frac{1}{1-\omega}}\,\int_0^{\infty} L^{-\gamma}(\theta) e^{(1-\omega)\int^{\theta}L(\eta)\,d\eta}\,d \theta,
\ege
where now the quantity $L(t)=1+\frac{N\beta \left(1-\frac{1}{m}\right)}{1+\frac{1}{m}\left(e^{\beta t}-1\right)}.$
\end{remark}

\begin{remark}\label{rem1a}
Assume now that
\bge\label{nk11}
0<\bar{u}_0<(\omega-1)^{\frac{1}{1-\omega}}\,I^{\frac{1}{1-\omega}}(\Sigma),
\ege
then $G(\Sigma)>0$ and since $G(\si)$ is strictly decreasing we get that $G(\si)>0$ for any $0<\si<\Sigma$ which implies that $F(\si)$ never blows up. Therefore,  since $F(\si)\leq \bar{u}(\si),$ there is still a possibility that $\bar{u}(\si)$ does not blow up either, however we cannot be sure and it remains to be verified numerically, see Section \ref{num-sec}.
\end{remark}

Next, we investigate the dynamics of some $L^\ell$-norms $||u(\cdot,\si)||_{\ell},$  which identify some invariant regions in the phase space. We first define
$\zeta(\si)=\avint u^r\,dx$,  $y(\si)=\avint u^{-p+1+r}\,dx$ and $w(\si)=\avint u^{p-1+r}\,dx$,
then H\"{o}lder's inequality implies
\bge
w(\si)y(\si)\geq \zeta^2(\si), \quad 0\leq \si<\Sigma. \label{psl1}
\ege
Our first result in this direction provides some conditions under which a finite-time blow-up takes place, when an {\it anti-Turing condition} is in place,  and is stated as follows.
\begin{theorem}\label{obu1}
Take $0<\gamma<1$ and $r\leq 1<\frac{p-1}{r}.$ Assume that $\Phi(\si), \Psi(\si)$ satisfy \eqref{mts2an}
then if one of the following conditions holds:
\begin{enumerate}
\item $w(0)<\frac{m_{\Psi}}{M_{\Phi}}\zeta(0)^{1-\gamma},$
\item $\frac{p-1}{r}\geq 2$ and $w(0)<1,$
\end{enumerate}
then finite-time blow-up occurs.
\end{theorem}
\begin{proof}
Set $\chi=u^{\frac{1}{\al}}$ with $\al\neq 0$, then following the same steps as in Proposition \ref{eiq} we derive
\bge
&&\quad \al \chi_{\si}=\alpha\left(\Delta \chi+4 (\al-1) \vert\nabla \chi^{\frac{1}{2}}\vert^2\right)-\Phi \chi+\Psi \frac{u^{p-1+\frac{1}{\al}}}{\left(\avint u^r\right)^{\ga}}, \;x\in \Omega_0,\; \si\in (0,\Sigma),\label{nob1} \\
&& \quad \frac{\pl \chi}{\pl \nu}=0 \quad \; x\in \pl\Omega_0,\; \si\in (0,\Sigma),\label{nob2}\\
&& \quad\chi(x,0)=u_0^{\frac{1}{\al}}(x), \quad x\in \Om_0.
 \label{nob3}
\ege
Averaging (\ref{nob1}) over $\Om_0$ and using zero-flux boundary condition \eqref{nob2}, we obtain
\bge\label{mkl1}
\alpha\frac{d}{d\si}\avint \chi=-4\alpha(1-\alpha)\avint\vert \nabla \chi^{\frac{1}{2}}\vert^2-\Phi(\si)\avint \chi+\Psi(\si)\frac{\avint u^{p-1+\frac{1}{\al}}}{\left(\avint u^r\right)^{\ga}}.
\ege
Relation \eqref{mkl1} for $\alpha=\frac{1}{r},$ since also $r\leq 1,$ entails that
\bge\label{psl2}
\frac{1}{r}\frac{d \zeta}{d\si}&&=-\frac{4}{r}\left(1-\frac{1}{r}\right)\avint\vert \nabla \chi^{\frac{1}{2}}\vert^2-\Phi(\si)\zeta(\si)+\Psi(\si)\frac{\avint u^{p-1+r}}{\left(\avint u^r\right)^{\ga}}\nonumber\\
&&\geq -M_{\Phi} \zeta(\si)+m_{\Psi}\frac{\zeta^{2-\ga}(\si)}{w(\si)}\geq \frac{\zeta(\si)}{w(\si)}\left(-M_{\Phi} w(\si)+m_{\Psi}\zeta^{1-\ga}(\si)\right),
\ege
which suffices by using  \eqref{psl1} together with \eqref{mts2an}. Furthermore, since $\frac{p-1}{r}>1$ then \eqref{mkl1} for $\alpha=\frac{1}{-p+1+r}$ leads to
\be\label{mk2}
\alpha\frac{dw}{d\si}=4\alpha(\alpha-1)\avint \vert \nabla u^{\frac{1}{2\alpha}}\vert^2-\Phi(\si)w+\Psi(\si)\zeta^{1-\gamma},
\ee
which, owing to \eqref{mts2an} and using the fact that $\alpha=\frac{1}{-p+1+r}<0$  entails
\begin{equation}
\frac{1}{p-1-r} \frac{dw}{d\si}\leq M_{\Phi} w(\si)-m_{\Psi}\zeta^{1-\gamma}(\si).
 \label{ob7}
\end{equation}
Note that since $0<\gamma<1$, we have  that the curve
$
\Gamma_1: w=\frac{m_{\Psi}\zeta^{1-\gamma}}{M_{\Phi}}, \ \zeta>0,
$
is concave in $w\zeta-$plane, with its endpoint at the origin $(0,0).$ Furthermore relations  (\ref{psl2}) and (\ref{ob7}) imply that the region $\mathcal{R}=\{ (\zeta,w) \mid w<\frac{m_{\Psi}\zeta^{1-\gamma}}{M_{\Phi}}\}$ is invariant, and $\zeta(\si)$ and $w(\si)$ are increasing and decreasing on $\mathcal{R},$ respectively. Under the assumption  $w(0)<\frac{m_{\Psi}\zeta^{1-\gamma}(0)}{M_{\Phi}}$ we have
$\frac{dw}{d\si}<0, \ \frac{d\zeta}{d\si}>0, \quad\mbox{for}\quad 0\leq \si<\Sigma, $
and thus,
\[ \frac{m_{\Psi}}{w(\si)}-\frac{M_{\Phi}}{\zeta^{1-\gamma}(\si)}\geq \frac{m_{\Psi}}{w(0)}-\frac{M_{\Phi}}{\zeta^{1-\gamma}(0)}\equiv c_0>0, \quad\mbox{for}\quad 0\leq \si<\Sigma.\]
Therefore by virtue of \eqref{psl2} we derive the differential inequality
\bge
\frac{1}{r}\frac{d\zeta}{d\si}\geq -M_{\Phi}\zeta(\si)+m_{\Psi}\frac{\zeta^{2-\gamma}(\si)}{w(\si)}&&=\zeta^{2-\gamma}(\si)\left(\frac{1}{w(\si)}-\frac{M_{\Phi}}{m_{\Psi}\zeta^{1-\gamma}(\si)}\right)\nonumber\\
&&\geq c_0\zeta^{2-\gamma}(\si), \quad 0\leq \si<\Sigma.
 \label{eqn:5.3}
\ege
Since $2-\gamma>1$, inequality (\ref{eqn:5.3}) implies that $\zeta(\si)$ blows up in finite time
$
\si_1\leq \hat{\si}_1\equiv\frac{\zeta^{\ga-1}(0)}{(1-\ga)c_0 r}<\infty,
$
and since $\zeta(\si)=\avint u^r\,dx\leq \|u(\cdot,\si)\|^r_{\infty}$
we conclude that $u(x,\si)$ blows up in finite time $\Sigma_b\leq \hat{\si}_1.$

We now consider  the latter case when $\frac{p-1}{r}\geq 2$ then $q=\frac{p-1-r}{r}\geq 1$, and thus by virtue of Jensen's inequality, \cite{ev10}, we obtain
$
\avint u^r\cdot\left( \avint (u^{-r})^q\right)^{\frac{1}{q}}\geq \avint u^{r}\cdot\avint u^{-r}\geq 1,
$
which entails $\zeta^{\frac{1}{r}}(\si)\geq w^{-\frac{1}{p-1-r}}(\si)$, and thus by virtue of \eqref{mts2an}
\bge\label{mk1}
w(\si)\geq \zeta^{-\frac{p-1-r}{r}(\si)}=\zeta^{1-\frac{p-1}{r}}(\si)> \frac{1}{\Phi(\si)}\zeta^{1-\frac{p-1}{r}}(\si)\geq \frac{m_{\Psi}}{M_{\Phi}}\zeta^{1-\frac{p-1}{r}}(\si),
\ege
for any $\si\in[0,\Sigma)$. Since $\frac{p-1}{r}\geq 2$, the curve $\Gamma_2: w=\frac{m_{\Psi}\zeta^{1-\frac{p-1}{r}}}{M_{\Phi}}, \ \zeta>0,$
is convex and approaches $+\infty$ and $0$ as $\zeta\downarrow 0^+$ and $\zeta\uparrow+\infty$, respectively. Moreover, the curves $\Gamma_1$ and $\Gamma_2$ intersect at the point  $(\zeta,w)=(1,1),$  and therefore, $w(0)<1$ combined with \eqref{mk1} implies that $w(0)<\frac{m_{\Psi}\zeta^{1-\gamma}(0)}{M_{\Phi}}$. Thus the latter case is reduced to the former one and again  finite-time blow-up for the solution $u(x,\si)$ is established.
\end{proof}
\begin{remark}\label{rem4}
Note that in the case of a stationary domain then $\zeta(\si)$ again blows up, see \cite{KS16, KS18}, in finite time
$
\si_2\leq \hat{\si}_2\equiv\frac{\zeta^{\ga-1}(0)}{(1-\ga)c_1 r},
$
where
$
c_1\equiv \frac{1}{w(0)}-\frac{1}{\zeta^{1-\gamma}(0)},
$
and thus $u(x,\si)$ blows in finite time $\Sigma_1\leq \hat{\si}_2$ under the condition $w(0)<\zeta(0)^{1-\gamma}.$
\end{remark}
\begin{remark}
For logistically growing or shrinking domain problem \eqref{nstregm1t}-\eqref{nstregm3t}  exhibit finite-time blow-up under the assumptions of Theorem \ref{obu1} whenever
$
w(0)<M_{L}^{-(\gamma+1)}\zeta(0)^{1-\gamma},
$
where
$
M_{L}:=\sup_{(0,\infty)} L(t)=\sup_{(0,\infty)}\left(1+\frac{N\beta \left(1-\frac{1}{m}\right)}{1+\frac{1}{m}\left(e^{\beta t}-1\right)}\right).
$
In particular, for a logistically growing domain, when $m>1,$ then $M_L=L(0)=1+N\beta\left(1-\frac{1}{m}\right),$ whilst for logistically shrinking domain, when $0<m<1$ we have $M_L=\lim_{t\to+\infty}L(t)=1$ and hence in that case blow-up conditions $(1)$ and $(2)$ of Theorem \ref{obu1} coincide with the ones of \cite[Theorem 3.5]{KS16}, see also Remark \ref{rem4}.
\end{remark}
Now we present a global-in-time existence result stated as follows.
\begin{theorem}\label{thm4}
Assume that $\frac{p-1}{r}<\min\{1, \frac{2}{N}, \frac{1}{2}(1-\frac{1}{r})\}$ and $0<\gamma<1.$ Consider functions $\Phi(\si), \Psi(\si)>0$ with
\bge\label{mts2a}
\inf_{(0,\Sigma)} \Phi(\si):=m_{\Phi}>0\quad\mbox{and}\quad \sup_{(0,\Sigma)} \Psi(\si):=M_{\Psi}<+\infty,
\ege
then problem \eqref{nstregm1k}-\eqref{nstregm3k} 
has a global-in-time solution.
 \label{thm:1.3}
\end{theorem}
\begin{proof}
We assume $\frac{p-1}{r}<\min\{ 1, \frac{2}{N}, \frac{1}{2}(1-\frac{1}{r})\}$ and $0<\gamma<1$. We also assume $N\geq 2$ since the complementary case $N=1$ is simpler.

Note that for $p>1$, we have $\frac{p-1}{r}<\frac{2}{N}$ and $r>p$. Therefore we have
$ 0<\frac{1}{r-p+1}<\min\left\{ 1, \frac{1}{p-1}\cdot\frac{2}{N-2}, \frac{1}{1-p+r\gamma}\right\},$
since $0<\gamma<1$. Chossing $\frac{1}{r-p+1}<\alpha<\min\{ 1, \frac{1}{p-1}\cdot\frac{2}{N-2}, \frac{1}{1-p+r\gamma} \}$, we derive
$ \max\left\{ \frac{N-2}{N}, \frac{1}{\alpha r}\right\}<\frac{1}{-\alpha+1+\alpha p}, $
and then we can find $\beta>0$ such that
\begin{equation}
\max\left\{ \frac{N-2}{N}, \frac{1}{\alpha r}\right\} <\frac{1}{\beta}<\frac{1}{-\alpha+1+\alpha p}<2,
 \label{eqn:15}
\end{equation}
which satisfies
\begin{equation}
\frac{\beta}{\alpha r}<1<\frac{\beta}{-\alpha+1+\alpha p}.
 \label{eqn:16}
\end{equation}
Recalling that $\chi=u^{\frac{1}{\alpha}}$ satisfies (\ref{nob1})-(\ref{nob3}) with $\displaystyle{\avint \frac{u^{p-1+\frac{1}{\alpha}}}{\left( \avint u^r\right)^\gamma}= \frac{\avint \chi^{-\alpha+1+\alpha p}}{\left( \avint \chi^{\alpha r}\right)^\gamma}},$
then by virtue of (\ref{eqn:16})
\[ \avint \chi^{-\alpha+1+\alpha p}\leq \left( \avint \chi^\beta\right)^{\frac{-\alpha+1+\alpha p}{\beta}}\quad\mbox{and}\quad \left( \avint \chi^{\alpha r}\right)^\gamma \geq \left( \avint \chi^\beta\right)^{\frac{\alpha r}{\beta}\cdot\gamma}, \]
thus we obtain the following estimate
\begin{equation}
\frac{\avint \chi^{-\alpha+1+\alpha p}}{\left( \avint\chi^{\alpha r}\right)^\gamma}\leq \left( \avint \chi^\beta\right)^{\frac{-\alpha+1+\alpha p-\alpha r\gamma}{\beta}}=\Vert \chi^{\frac{1}{2}}\Vert_{2\beta}^{2(1-\lambda)},
 \label{eqn:3.3}
\end{equation}
with $0<\lambda=\alpha\{1-p+r\gamma\}<1$, recalling that $\frac{p-1}{r}<\gamma$ and $\alpha<\frac{1}{1-p+r\gamma}$.
Averaging (\ref{nob1}) over $\Om_0$ leads to the following,
\begin{equation}\label{ob4}
\alpha\frac{d}{d\si}\avint \chi+4\alpha(1-\alpha)\avint\vert \nabla \chi^{\frac{1}{2}}\vert^2+\Phi(\si)\avint \chi=\Psi(\si)\frac{\avint \chi^{-\alpha+1+\alpha p}}{\left( \avint\chi^{\alpha r}\right)^\gamma},
\end{equation}
and hence
\[ \frac{d}{d\si}\avint \chi+4(1-\alpha)\avint\vert\nabla \chi^{\frac{1}{2}}\vert^2+\frac{m_{\Phi}}{\alpha}\avint \chi\leq \frac{M_{\Psi}}{\alpha}\Vert \chi^{\frac{1}{2}}\Vert_{2\beta}^{2(1-\lambda)}, \]
by virtue of \eqref{mts2a}, \eqref{eqn:15} and \eqref{eqn:3.3}. Now since $1<2\beta<\frac{2N}{N-2}$ holds due to (\ref{eqn:15}) and applying first Sobolev's and then Young's inequalities we derive
\[ \frac{d}{d\si}\avint\chi+c\Vert \chi^{\frac{1}{2}}\Vert_{H^1}^2+\frac{M_{\Psi}}{\alpha}\avint \chi\leq C, \]
which implies $\avint \chi\leq C.$
Since $\frac{1}{\alpha}$ can be chosen to be close to $r-p+1$, the above estimate gives
\begin{equation}
\Vert u(\cdot,\si)\Vert_q\leq C_q, \quad\mbox{for any}\quad 1\leq q<r-p+1,
 \label{eqn:15h}
\end{equation}
recalling that $\chi=u^{\frac{1}{\alpha}}.$ Note that $\frac{p-1}{r}<\frac{1}{2}(1-\frac{1}{r})$ implies $\frac{r-p+1}{p}>1$ and thus we obtain global-in-time existence
by using the same bootstrap argument as in \cite[Theorem 3.4]{KS16}.
\end{proof}
\begin{remark}
Note that condition \eqref{mts2a} is satisfied  in the case of an exponential shrinking domain as indicated in Remark \ref{aal5},  see in particular \eqref{aal3} and \eqref{aal4}.
\end{remark}
\section{Turing Instability and Pattern Formation}\label{tipf}
In the current section, we present a Turing-instability result for problem \eqref{nstregm1k}-\eqref{nstregm3k}, restricting ourselves to the radial case $\Om_0=B_1(0):=\{x\in \R^N \mid \vert x\vert<1\}.$ Then the solution of $u$ \eqref{nstregm1k}-\eqref{nstregm3k} is radially symmetric, i.e. $u(x,t)=u(R,t)$ for $R=|x|,$ and thus it satisfies the following
 \bge
 && u_{\si}=\Delta_R u-\Phi(\si)u+\displaystyle\frac{\Psi(\si)u^p}{\left(\avint u^r\right)^{\ga}}, \quad R\in (0,1),\; \si\in (0,\Si), \label{rnstregm1}\\
&& u_R(0,\si)=u(1,t)=0, \quad  t\in (0,\Si), \label{rnstregm2}\\
 && u(R,0)=u_0(R),\quad 0<R<1,  \label{rnstregm3}
\ege
where $\Delta_R u:=u_{RR}+\frac{N-1}{R}u_R.$

It can be seen, see also \cite{KS16, KS18}, that under the Turing condition \eqref{tc},
the spatial homogeneous solutions of \eqref{nstregm1k}-\eqref{nstregm3k}, i.e. the solution of the problem
$\frac{du}{d\si}=-\Phi(\si)u+\Psi(\si)u^{p-r\gamma}, \quad \left. u\right\vert_{\si=0}=\bar{u}_0>0, $
never exhibits blow-up, as long as $\Phi(\si), \Psi(\si)$ are both bounded, since the nonlinearity $f(u)=u^{p-r\gamma}$ is sub-linear. On the other hand, considering spatial inhomogeneous solutions of \eqref{nstregm1k}-\eqref{nstregm3k} we will show below,  see Theorem \ref{thm:6.1}, that a diffusion-driven instability phenomenon occurs when spiky type of initial conditions are considered.
Indeed, next we consider an initial datum of the form, \cite{hy95},
\begin{equation}
u_0(R)=\lambda\psi_\delta(R),
 \label{eqn:6.1}
\end{equation}
with $0<\lambda\ll 1$ and
\begin{equation}\label{id}
\psi_\delta(R)=\begin{cases}
      R^{-a},&  \de\leq R\leq 1, \\
      \de^{-a}\left(1+\frac{a}{2}\right)-\frac{a}{2}\de^{-(a+2)}R^2, & 0\leq R<\de,
      \end{cases}
\end{equation}
where $a=\frac{2}{p-1}$ and $0<\de<1$. It is easily seen, \cite{KS16, KS18}, that for $\psi_{\delta}$ the following lemma holds.
\begin{lemma}\label{kkl1}
For the function $\psi_{\delta}$ defined by \eqref{id} we have:
\begin{enumerate}
\item[(i)] For any $0<\delta<1,$ there holds in a weak sense
\begin{equation}
\displaystyle{ \Delta_R \psi_\delta\geq -Na\psi_\delta^p}.
 \label{eqn:6.3}
\end{equation}
\item[(ii)] If $m>0$ and $N>ma$, we have
\begin{equation}\label{ine}
\avint \psi_\delta^m=\frac{N}{N-ma}+O\left(\de^{N-ma}\right), \quad \delta\downarrow 0.
\end{equation}
\end{enumerate}
\end{lemma}
Now, if we consider
\begin{equation}
\mu>1+r\gamma
 \label{eqn:6.15}
\end{equation}
and set
$\alpha_1=\sup_{0<\delta<1}\frac{1}{\bar{\psi}_\delta^\mu}\avint\psi_\delta^p$, and $\alpha_2=\inf_{0<\delta<1}\frac{1}{\bar{\psi}_\delta^\mu}\avint\psi_\delta^p$,
then since $p>\frac{N}{N-2}$, relation(\ref{ine}) is applicable for $m=p$ and $m=1$, and thus owing to \eqref{eqn:6.15} we obtain
\bge\label{at1} 0<\alpha_1, \alpha_2<\infty. \ege
Furthermore, it follows that
\begin{equation}
d\equiv\inf_{0<\delta<1}\left( \frac{1}{2\alpha_1}\right)^{\frac{r\gamma}{p}}\left( \frac{1}{2\bar{\varphi}_\delta}\right)^{\frac{r\gamma}{p}\mu}>0,
 \label{eqn:6.7}
\end{equation}
and the initial data $u_0$ defined by \eqref{eqn:6.1} also satisfies the following lemma, see also \cite{KS16, KS18},
\begin{lemma}
If $p>\frac{N}{N-2}$ and $\frac{p-1}{r}<\gamma$, there exists $\lambda_0=\lambda_0(d)>0$ such that for any $0<\lambda\leq \lambda_0$ there holds
\begin{equation}
\Delta_R u_0+d\lambda^{-r\gamma}u_0^p\geq 2u_0^p.
 \label{eqn:6.4}
\end{equation}
 \label{lem:6.2}
\end{lemma}
Hereafter, we fix $0<\lambda\leq \lambda_0=\lambda_0(d)$ so that (\ref{eqn:6.4}) is satisfied.  Given $0<\delta<1$, let $\Si_\delta>0$ be the maximal existence time of the solution to (\ref{rnstregm1})-(\ref{rnstregm3}) with initial data of the form \eqref{eqn:6.1}-\eqref{id}.
Next, we introduce the new variable $z=e^{\int^{\si}\Phi(s)\,ds }u,$ so that the linear dissipative term $-\Phi(\si) u$ is eliminated and then $z$ satisfies
\bge
&&z_\si= \Delta_R z +K(\si)z^p,\quad R\in (0,1),\; \si\in (0,\Si_\delta),\label{tbsd3}\\
&&z_R(0,\si)=z_R(1,\si)=0,\quad  \si\in (0,\Si_\delta),\label{tbsd3a}\\
&& z(R,0)=u_0(R),\quad 0<R<1, \label{tbsd3b}
 \ege
where
\begin{equation} \label{nnt}
K(\si)=\frac{\Psi(\si)e^{(1+r\gamma-p)\int^\si \Phi(s)\,ds}}{\Big(\displaystyle\avint z^r \Big)^{\gamma}}.
\end{equation}
It is clear that as long as $\Phi(\si)$ is bounded then $u$ blows-up in finite time if and only if $z$ does so.
Assuming now that both $\Phi(\si)$ and $\Psi(\si)$ are positive and bounded,
 which is the case for the evolution provided by $\psi(\si)$ satisfying \eqref{diq1} or for an exponential shrinking domain as indicated in Remarks \ref{rem1} and \ref{aal5}, then by virtue of (\ref{eqn:1.17}) we have
\begin{equation}
0< K(\si)=\frac{\Psi(\si)e^{(1-p)\int^\si \Phi(s)\,ds}}{\Big(\displaystyle\avint u^r \Big)^{\gamma}}\leq C<\infty,
 \label{eqn:6.8}
\end{equation}
thus averaging of (\ref{tbsd3}) entails
\begin{equation}
\frac{d\bar{z}}{d\si}=K(\si)\avint z^p,
 \label{eqn:6.8h}
\end{equation}
and thus \eqref{eqn:6.8} yields
\begin{equation}
\bar{z}(\si)\geq \bar{z}(0)=\bar{u}_0:=\avint u_0.
 \label{eqn:6.9}
\end{equation}
Henceforth, the positivity and the boundedness of $\Phi(\si), \Psi(\si)$ as well as the Turing condition \eqref{tc} are imposed. Moving towards the proof of Theorem \ref{thm:6.1} we first need to establish some auxiliary results.
Next, we provide a useful estimate of $z$ that will be frequently used throughout the text.
\begin{lemma}\label{lem3}
The solution $z$ of problem \eqref{tbsd3}-\eqref{tbsd3b} satisfies
\bge
&& R^N z(R,\si)\leq \bar{z}(\si) \quad\mbox{in}\quad(0,1)\times (0,\Si_{\de}),
 \label{eqn:6.10}\\
&& \text{and} \notag \\
&& z_R\left(\frac{3}{4},\si\right)\leq -c, \ \ 0\leq \si<\Si_\delta, \label{eqn:6.11}
\ege
for any $0<\delta<1$ and some positive constant $c.$
\end{lemma}
\begin{proof}
Let us define $w=R^{N-1}z_{R}$, then it is easily checked that $w$ satisfies
 $\mathcal{H}[w]=0, \quad\mbox{for}\quad (R,\si)\in (0,1)\times (0,\Si_{\de})$, with 
 $w(0,\si)=w(1,\si)=0$, for $\si\in(0,\Si_{\de})$, and 
$w(R,0)<0$, for $0<R<1$, 
where
$\mathcal{H}[w]\equiv w_\si-w_{RR}+\frac{N-1}{\rho}w_R-p K(\si)z^{p-1}w$.
Owing to the maximum principle, and recalling that $K(\si)$ is bounded by \eqref{eqn:6.8}, we get that $w\leq 0$, which implies  $z_{R}\leq 0$ in $(0,1)\times (0,\Si_{\de})$. Accordingly, inequality (\ref{eqn:6.10}) follows since
\begin{eqnarray*}
R^N z(R,\si) & = & z(R,\si)\int_0^R N s^{N-1}ds\leq \int_0^R Nz(s,\si)s^{N-1}\,ds \\
& \leq & \int_0^1Nz(s,\si)s^{N-1}\,ds=\avint z=\bar{z}(\si).
\end{eqnarray*}
Now given that $w\leq 0$ together with \eqref{eqn:6.8} we have
\[ w_\si-w_{RR}+\frac{N-1}{\rho}w_R=pK(\si)z^{p-1}w\leq 0 \quad\mbox{in}\quad (0,1)\times (0,\Si_{\de}),\]
with $w\left(\frac{1}{2},\si\right)\leq0,\quad w\left(1,\si\right)\leq 0$, for $\si\in (0,\Si_{\de})$, and
$w(R,0)=\rho^{N-1}u'_{0}(R)\leq -c$, for $\frac{1}{2}<\rho<1$,
which implies $w\leq -c$ in $(\frac{1}{2},1)\times (0, \Si_\delta)$, and thus (\ref{eqn:6.11}) holds.
\end{proof}
\begin{lemma}\label{kkl2}
Take $\varepsilon>0$ and $1<q<p$ then $\vartheta$ defined as
\be\label{ik1}
 \vartheta:=R^{N-1}z_R+\varepsilon\cdot\frac{R^Nz^q}{\bar{z}^{\gamma+1}},
\ee
satisfies
\begin{eqnarray}
 \qquad\mathcal{H}[\vartheta]\leq -\frac{2q\varepsilon}{\bar{z}^{\gamma+1}}z^{q-1}\vartheta &&+\frac{\varepsilon R^Nz^q}{\bar{z}^{2(\gamma+1)}}\left\{2q\varepsilon z^{q-1}\right.\no\\
&&\left. -m_{\Psi}(\gamma+1)\bar{z}^{\gamma-r\gamma}\avint z^p-m_{\Psi}(p-q)z^{p-1}\bar{z}^{\gamma+1-r\gamma}\right\},
 \label{eqn:6.13h}
\end{eqnarray}
for  $(R,\si)\in (0,1)\times (0,\Si_{\de}),$ where $m_{\Psi}=\inf_{\si\in (0,\Si_\de)}\Psi(\si)>0.$
 \label{lem:6.4}
\end{lemma}
\begin{proof}
It is readily checked that
$\mathcal{H}\left[R^{N-1} z_{R}\right]=0,$
while by straightforward calculations we derive
\bge\label{tbsd6}
\mathcal{H}\left[\vep R^{N} \frac{z^q}{\ol{z}^{\gamma+1}}\right]&&=\frac{2q(N-1)\vep R^{N-1}z^{q-1}}{\ol{z}^{\gamma+1}}z_{R}+\frac{q\vep R^N z^{p-1+q}}{\ol{z}^{\gamma+1}} K(\si)-\frac{(\gamma+1) \vep R^N z^q}{\ol{z}^{\gamma+2}}\,K(\si)\,\avint z^p\,dx\no\\
&&-\frac{2 q N\vep R^{N-1} z^{q-1}}{\ol{z}^{\gamma+1}}z_R-\frac{q(q-1)\vep R^N z^{q-2}}{\ol{z}^{\gamma+1}}z_R^2-\frac{p\vep R^N z^{p-1+q}}{\ol{z}^{\gamma+1}} K(\si)\no\\
&&\leq-\frac{2 q \vep z^{q-1}}{\ol{z}^{\gamma+1}}\vartheta+\frac{\vep R^N z^q}{\ol{z}^{2(\gamma+1)}}\left[2q\vep z^{q-1}-\frac{\Psi(\si)(\gamma+1) \ol{z}^\ga e^{(1+r\ga-p)\int^\si\Phi(s)\,ds}}{\left(\avint z^r\,dx\right)^\ga}\avint z^p\,dx\right.\notag \\
&& \left. \quad - \frac{\Psi(\si)(p-q) z^{p-1} \ol{z}^{\ga+1}}{\left(\avint z^r\,dx\right)^\ga} e^{(1+r\ga-p)\int^\si \Phi(s)\,ds} \right].
\ege
Then by virtue of the H\"{o}lder's inequality, and since $1\leq r\leq p,$ \eqref{tbsd6} entails the desired estimate \eqref{eqn:6.13h}.
\end{proof}

Next note that when $p>\frac{N}{N-2}$, there is $1<q<p$ such that $N>\frac{2p}{q-1}$ and thus the following quantities
\begin{equation}
A_1\equiv\sup_{0<\de<1}\frac{1}{\bar{u}_0^{\mu}} \avint u_0^p=\lambda^{\mu-p}\alpha_1\quad\mbox{and}\quad A_2\equiv\inf_{0<\de<1}\frac{1}{\bar{u}_0^{\mu}} \avint u_0^p=\lambda^{\mu-p}\alpha_2,
 \label{eqn:6.14}
\end{equation}
are finite due to \eqref{at1}.

An essential ingredient for the proof of Theorem  \ref{thm:6.1} is the following  key estimate of the $L^p-$norm of $z$ in terms of $A_1$ and $A_2.$
\begin{proposition}\label{lem4}
There exist $0<\delta_0<1$ and $0<\si_0\leq 1$ independent of any $0<\delta\leq \delta_0,$ such that the following estimate is satisfied
\begin{equation}\label{lue}
\frac{1}{2} A_2\ol{z}^{\mu}\leq  \avint z^p\,dx\leq 2 A_1 \ol{z}^{\mu},
\end{equation}
for any $0<\si<\min\{\si_0,\Si_{\de}\}.$
\end{proposition}
The proof of Proposition \ref{lem4} requires some further auxiliary results provided below. Let us define $0<\si_0(\delta)<\Si_\delta$ to be the maximal time for which inequality (\ref{lue}) is valid in $0<\si<\si_0(\delta),$ then we have
\begin{equation}
\frac{1}{2}A_2\bar{z}^\mu\leq \avint z^p\leq 2A_1\bar{z}^\mu.
 \label{eqn:6.18}
\end{equation}
We only regard the case $\si_0(\delta)\leq 1,$ since otherwise there is nothing to prove. Then the following lemma holds true.

\begin{lemma}\label{nkl2}
There exists $0<\si_1<1$ such that
\begin{equation}
\bar{z}(\si)\leq 2\bar{u}_0, \quad 0<\si<\min\{ \si_1, \si_0(\delta)\},
 \label{eqn:6.18h}
\end{equation}
for any $0<\delta<1$.
 \label{lem:6.6}
\end{lemma}

\begin{proof}
Since $r\geq 1$ and $\si_0(\delta)\leq 1$, then by virtue of (\ref{nnt}) and (\ref{eqn:6.8h})
$\frac{d\bar{z}}{d\si}\leq 2A_1 M_{\Psi} e^{(1+r\gamma-p)M_{\Phi}}\bar{z}^{\mu-r\gamma}$, for $0<\si<\si_0(\delta)$,
recalling that $M_{\Phi}=\sup_{\si\in(0,\Si_{\de})}\Phi(\si)<+\infty$ and $M_{\Psi}=\sup_{\si\in(0,\Si_{\de})}\Psi(\si)<+\infty.$

Setting $C_1=2A_1 M_{\Psi} e^{(1+r\gamma-p)M_{\Phi}}$ and taking into account (\ref{eqn:6.15}) we then derive
$\ol{z}(\si)\leq \left[\bar{u}_0^{1+r\ga-\mu}-C_1(\mu-r\ga-1)\si\right]^{-\frac{1}{\mu-r\ga-1}}. $
 Accordingly, (\ref{eqn:6.18h}) holds for any $0<\si<\min\{ \si_1, \si_0(\delta)\}$ where $\si_1$ is independent of any $0<\delta<1$ and estimated as
$
\si_1\leq\min\left\{\frac{1-2^{1+r\ga-\mu}}{C_1(\mu-r\ga-1)}\ol{u}_0^{1+r\ga-\mu},1\right\}.
$
\end{proof}
Another fruitful estimate is provided by the next lemma.
\begin{lemma}\label{nkl3}
There exist $0<\delta_0<1$ and $0<R_0<\frac{3}{4}$ such that for any $0<\delta\leq \delta_0$  the following estimate is valid
\begin{equation}
\frac{1}{\vert\Omega\vert}\int_{B_{R_0}(0)}z^p\leq \frac{A_2}{8}\bar{z}^\mu, \quad\mbox{for}\quad 0<\si<\min\{ \si_1, \si_0(\delta)\},
 \label{eqn:6.19}
\end{equation}
where $B_{R_0}(0)=\{x\in \R^N \mid \vert x\vert<R_0\}.$
\end{lemma}
\begin{proof}
By virtue of (\ref{eqn:6.9}) and (\ref{eqn:6.18h}) it follows that
\begin{equation}
\bar{u}_0\leq \bar{z}(\si)\leq 2\bar{u}_0, \quad\mbox{for}\quad 0<\si<\min\{\si_1, \si_0(\delta)\}.
 \label{eqn:6.20}
\end{equation}
Furthermore,  we note that the growth of $\avint z^p$ is controlled by the estimate (\ref{lue}) for $0<\min\{ \si_1, \si_0(\delta)\}$ and since $p>q$ then Young's inequality ensures that the second term of the right-hand side in (\ref{eqn:6.13h}) is negative for $0<\si<\min\{ \si_1, \si_0(\delta)\}$, uniformly in $0<\delta<1$, provided that $0<\varepsilon\leq \varepsilon_0$ for some $0<\varepsilon_0\ll 1.$ Therefore
\begin{equation}
\mathcal{H}[\vartheta]\leq
-\frac{2q\vep z^{q-1}}{\bar{z}^{\gamma+1}}
\vartheta \quad\mbox{in}\quad (0,1)\times (0,\min\{ \si_1, \si_0(\delta)\}).
 \label{eqn:6.21}
\end{equation}
Moreover (\ref{eqn:6.10}) and (\ref{eqn:6.20}) imply
\begin{eqnarray*}
\vartheta(R,\si) & = & R^{N-1}z_R+\varepsilon\cdot\frac{R^Nz^q}{\bar{z}^{\gamma+1}} \leq R^{N-1}z_R+\varepsilon\cdot R^{N(1-q)}\bar{z}^{q-\gamma-1} \\
& \leq & R^{N-1}z_R+C\cdot\varepsilon R^{N(1-q)}\quad\mbox{in}\quad (0,1)\times (0,\min\{ \si_1, \si_0(\delta)\}),
\end{eqnarray*}
which, for $0<\varepsilon\leq \varepsilon_0,$ entails
\begin{equation}
\vartheta\left(\frac{3}{4},\si\right)<0, \quad 0<\si<\min\{ \si_1, \si_0(\delta)\},
 \label{eqn:6.22}
\end{equation}
owing to (\ref{eqn:6.11}) and provided that $0<\varepsilon_0\ll 1.$ Additionally \eqref{ik1} for $t=0$ gives
\begin{equation}
\vartheta(R,0) = R^{N-1}\left(\lambda \psi_{\delta}'(R)+\varepsilon\lambda^{q-\gamma-1}R\cdot\frac{\psi_\delta^q}{\bar{\psi}_\delta^{\gamma+1}}\right).
 \label{eqn:6.23}
\end{equation}
For $0\leq R<\delta$ and $\varepsilon$ small enough and independent of $0<\de<\de_0,$ then the right-hand side of (\ref{eqn:6.23}) can be estimated as
\bgee
R^{N}\lambda\left( -a\delta^{-a-2}+\varepsilon\lambda^{q-\gamma-2}\cdot \frac{\psi_\delta^q}{\bar{\psi}_\delta^{\gamma+1}} \right)\lesssim R^{N}\lambda\left( -a\delta^{-a-2}+\varepsilon\lambda^{q-\gamma-2}\cdot \delta^{-aq} \right)\lesssim 0,
\egee
since by virtue of (\ref{id}) and (\ref{ine}) and for $m=1,$ there holds
$\displaystyle\frac{\psi_\delta^q}{\bar{\psi}_\delta^{\gamma+1}}\lesssim \delta^{-aq},\; \delta\downarrow 0,$ uniformly in $0\leq R<\delta,$
taking also into account that $a+2=ap>ak.$

On the other hand, for $\delta\leq R\leq 1$ and by using (\ref{ine}) for $m=1$ we take
\begin{equation}
\vartheta(R,0)=R^N\lambda\left(-a R^{-a-1}+\varepsilon\lambda^{q-\gamma-1}\frac{R^{1-aq}}{\bar{\psi}_R^{\gamma+1}}\right),
 \label{eqn:6.24}
\end{equation}
which, since $a+2=ap>aq$ implies $-a-1<-aq+1$, finally yields
$ \vartheta(R,0)<0, \quad \de\leq R\leq \frac{3}{4},$
for any $0<\delta\leq \delta_0$ and $0<\varepsilon\leq \varepsilon_0$, provided $\varepsilon_0$ is chosen sufficiently small.
Accordingly, we derive
\begin{equation}
\vartheta(R,0)<0, \quad 0\leq R\leq \frac{3}{4},
 \label{eqn:6.25}
\end{equation}
for any $0<\delta\leq \delta_0$ and $0<\varepsilon\leq \varepsilon_0$, provided $0<\varepsilon_0\ll 1.$

In conjunction of (\ref{eqn:6.21}), (\ref{eqn:6.22}) and (\ref{eqn:6.25}) we deduce
$ \vartheta(R,\si)=R^{N-1}z_R+\varepsilon\cdot\frac{R^Nz^q}{\bar{z}^{\gamma+1}}\leq 0$ in $(0,\frac{3}{4})\times (0,\min\{ \si_1, \si_0(\delta)\})$,  and finally
\begin{equation}
z(R,\si)\leq \left( \frac{\varepsilon}{2}(q-1)\right)^{-\frac{1}{q-1}}\cdot R^{-\frac{2}{q-1}}\cdot\bar{z}^{\frac{\gamma-1}{q-1}}(\si) \quad \mbox{in $(0,\frac{3}{4})\times (0,\min\{ \si_1, \si_0(\delta)\})$}.
 \label{eqn:6.26}
\end{equation}
Note that owing to  $N>\frac{2p}{q-1}$ there holds $-\frac{2}{q-1}\cdot p+N-1>-1$ and thus (\ref{eqn:6.19}) is valid for some $0<R_0<\frac{3}{4}.$
\end{proof}
\begin{remark}\label{nkl1a}
Estimate \eqref{eqn:6.26} entails that $z(R,\si)$ can only blow-up  in the origin $R=0;$ that is, only a single-point blow-up is feasible.
\end{remark}
Next we prove the key estimate \eqref{lue} using essentially Lemmas \ref{nkl2} and \ref{nkl3}.

\begin{proof}[Proof of Proposition \ref{lem4}]
By virtue of (\ref{eqn:6.15}) and since $\frac{p-1}{r}<\delta$, there holds that $\ell=\frac{\mu}{p}>1.$ We can easily check, \cite{KS16,KS18}, that $\theta=\displaystyle{\frac{z}{\ol{z}^{\ell}}}$ satisfies
\[\theta_{\si}=\Delta_R \theta+ \Psi(\si)e^{(r\ga+1-p)\int^\si \Phi(s)\,ds}\left[\frac{z^p}{\ol{z}^{\ell}\left(\avint z^r\right)^{\ga}}-\frac{\ell z\avint z^p}{\ol{z}^{\ell+1}\left(\avint z^r\right)^{\ga}} \right], \]
in $\Omega_0\times(0, \min\{ \si_0, \Si_\delta\})$, with $\frac{\partial\theta}{\partial\nu}=0$, on $\partial\Omega_0\times(0, \min\{ \si_0, \Si_\delta\})$, and $\theta(x,0)=\frac{z(x,0)}{\bar{z}_0^{\ell}}$, on $\Omega_0$.
In conjunction with (\ref{eqn:1.17}), (\ref{eqn:6.9}), (\ref{eqn:6.10}), (\ref{eqn:6.18}), and (\ref{eqn:6.18h}) we deduce that
\bge\label{kik1}
\left\Vert \theta,\ \frac{z^p}{\ol{z}^{\ell}\left(\avint z^r\right)^{\ga}}, \ \frac{\ell z\avint z^p}{\ol{z}^{\ell+1}\left(\avint z^r\right)^{\ga}}\right\Vert_{L^\infty((\Omega_0\setminus B_{R_0}(0))\times \min\{ \si_1, \si_0(\delta)\})} <+\infty,
\ege
uniformly in $0<\delta\leq \delta_0,$ and using the fact that $\Psi(\si)$ and $\Psi(\si)$ are both bounded and positive.
Estimate \eqref{kik1} by virtue of the standard parabolic regularity, see DeGiorgi-Nash-Moser estimates in \cite[pages 144-145]{Lie96}, entails the existence of $0<\si_2\leq \si_1$ independent of $0<\de\leq \de_0$:
$ \sup_{0< \si<\min\{\si_2, \si_0(\delta)\}}\left\Vert \theta^p(\cdot,\si)-\theta^p(\cdot,0)\right\Vert_{L^1(\Omega_0\setminus B_{R_0}(0)}\leq \frac{A_2}{8}\vert\Omega_0\vert, $
which yields
\begin{equation}
\left\vert \frac{1}{\vert\Omega_0\vert}\int_{\Omega_0\setminus B_{R_0}(0)}\frac{z^p}{\bar{z}^\mu}-\frac{1}{\vert\Omega_0\vert}\int_{\Omega_0\setminus B_{R_0}(0)}\frac{z_0^p}{\bar{z}_0^\mu} \right\vert \leq \frac{A_2}{8},
 \label{eqn:6.28}
\end{equation}
with $0<\si<\min\{ \si_2, \si_0(\delta)$ for any $0<\delta\leq \delta_0$. Combining  (\ref{eqn:6.19}) and (\ref{eqn:6.28}) we deduce
$ \left\vert \avint \frac{z^p}{\ol{z}^\mu}-\avint \frac{z_0^p}{\ol{z}_0^\mu}\right\vert\leq \frac{3 A_2}{8}, \;\mbox{for}\; 0<\si<\min\{ \si_2, \si_0(\delta)\}\;\mbox{and}\; 0<\delta\leq \delta_0,$
and thus we finally obtain
\begin{equation}\label{tbsd13}
\frac{5A_2}{8}\leq \avint \frac{z^p}{\ol{z}^{\mu}} \leq \frac{11 A_1}{8}, \quad 0<\si<\min\{ \si_2, \si_0(\de)\}, \ 0<\delta\leq \delta_0,
\end{equation}
taking also into consideration
$ A_2\leq \avint\frac{z_0^p}{\bar{z}_0^\mu}\leq A_1. $
Consequently, if we consider $\si_0(\delta)\leq \si_2$ then it follows
$ \frac{1}{2}A_2\bar{z}^\mu<\frac{5}{8}A_2\bar{z}^\mu\leq \avint z^p\leq \frac{11}{8}A_1\bar{z}^\mu<2A_1\bar{z}^\mu$, for $ 0<\si<\si_0(\delta)$,
and thus a continuity argument implies that
$\frac{1}{2}A_2\bar{z}^\mu\leq \avint z^p\leq 2A_1\bar{z}^\mu$,  with $0<\si<\si_0(\delta)+\eta$,
for some $\eta>0,$ which contradicts the definition of $\si_0(\delta)$. Accordingly, we derive that $\si_2<\si_0(\delta)$ for any $0<\delta\leq \delta_0$, and the proof of Proposition \ref{lem4} is complete for $\si_0=\si_2.$
\end{proof}
Next we present the Turing instability result intimated at the beginning of the section, which in particular is exhibited in the form of diffusion-driven blow-up.
\begin{theorem}\label{thm5}
Consider $N\geq 3,\;1\leq r\leq p$, $p>\frac{N}{N-2}, \frac{2}{N}<\frac{p-1}{r}<\gamma$ and $\gamma>1.$ Assume that both $\Phi(\si)$ and $\Psi(\si)$ are positive and bounded.  Then there exists $\lambda_0>0$ such that for any $0<\lambda\leq \lambda_0$ there exists $0<\delta_0=\delta_0(\lambda)<1$ and any solution of problem (\ref{rnstregm1})-(\ref{rnstregm3}) with initial data of the form (\ref{eqn:6.1}) and $0<\delta\leq \delta_0$ blows up in finite time.
 \label{thm:6.1}
\end{theorem}
\begin{proof}
First note that $\si_0\leq \si_1$ in (\ref{eqn:6.18h}), then from (\ref{eqn:6.7}) and (\ref{eqn:6.14}) we have
\begin{eqnarray}
\quad K(\si) \geq
\frac{m_{\Psi}}{\left( \avint z^p\right)^{\frac{r\gamma}{p}}}\geq
 m_{\Psi}\left( \frac{1}{2\alpha_1}\right)^{\frac{r\gamma}{p}}\cdot\left( \frac{1}{2\bar{\psi}_\delta}\right)^{\frac{r\gamma}{p}\mu}\lambda^{-r\gamma}\geq m_{\Psi}d\lambda^{-r\gamma}\equiv D,
 \label{eqn:6.33}
\end{eqnarray}
for $0<\si<\min\{ \si_0, \Si_\delta\}.$ Note also that for $0<\lambda\leq \lambda_0(d)$, then inequality (\ref{eqn:6.4}) entails
\begin{equation}
\Delta u_0+Du_0^p\geq 2u_0^p
 \label{eqn:6.34}
\end{equation}
for any $0<\delta\leq \delta_0$. The comparison principle in conjunction with (\ref{eqn:6.33}) and (\ref{eqn:6.34}) then yields
\begin{equation}
z\geq \tilde{z} \quad\mbox{in}\quad Q_0\equiv\Omega_0\times (0, \min\{ \si_0, \Si_\delta\}),
 \label{eqn:6.35}
\end{equation}
where $\tilde z=\tilde z(x,t)$ solves the following partial differential equation
\bge
&&\tilde{z}_{\si}=\Delta \tilde{z}+D\tilde{z}^p, \quad\mbox{in}\quad Q_0,\label{eqn:6.36}\\
&&\frac{\partial \tilde{z}}{\partial\nu}=0,\quad\mbox{on}\quad\partial\Omega_0\times(0, \min\{ \si_0, \Si_\delta\}),\\
&&\tilde{z}(|x|,\si)=u_0(\vert x\vert)\quad\mbox{in}\quad \Om_0.
 \label{eqn:6.36}
\ege
Setting $h(x,\si):=\tilde{z}_{\si}(x,\si)-\tilde{z}^p(x,\si),$
then
\begin{eqnarray*}
h_{\si} = \Delta h+p(p-1) \tilde{z}^{p-2} |\nabla \tilde{z}|^2+ D p \tilde{z}^{p-1}\,h
 \geq  \Delta h+ D p \tilde{z}^{p-1}\,h \quad\mbox{in}\quad Q_0,
\end{eqnarray*}
with
\bgee
h(x,0)= \Delta \tilde{z}(x,0)+D\tilde{z}^p(x,0)-\tilde{z}^p(x,0)=\Delta u_0+(D-1)u_0^p\geq u_0^p>0,\quad\mbox{in}\quad \Om_0,
\egee
whilst
$
\frac{\partial h}{\partial \nu}=0\;\mbox{on}\;\partial\Omega_0\times(0, \min\{ \si_0, \Si_\delta\}).
$
Therefore, owing to the maximum principle, we derive
$
\tilde{z}_{\si}>\tilde{z}^p\quad\mbox{in}\quad Q_0,
$
and thus via integration we obtain
\[ \tilde{z}(0,\si)\geq\left(\frac{1}{z_0^{p-1}(0)}-(p-1)\si\right)^{-\frac{1}{p-1}}=\left\{\left(\frac{\de^{a}}{\la(1+\frac{a}{2})}\right)^{p-1}-(p-1)\si\right\}^{-\frac{1}{p-1}}
\]
for $0<\si<\min\{\si_0, \Si_\delta\}$, and therefore,
\begin{equation}
\min\{\si_0, \Si_\delta\}<\frac{1}{p-1}\cdot
\left( \frac{\delta^a}{\lambda(1+\frac{a}{2})}\right)^{p-1}.
 \label{eqn:6.38}
\end{equation}
Note that for $0<\delta\ll 1$, then the right-hand side on (\ref{eqn:6.38}) is less than $\si_0$, so
$\Si_\delta<\frac{1}{p-1}\cdot
\left( \frac{\delta^a}{\lambda(1+\frac{a}{2})}\right)^{p-1}<+\infty$.
\end{proof}
\begin{remark}
Recalling that $z(x,\si)=e^{\int^{\si}\Phi(s)\,ds} u(x,\si)$ we also obtain the occurrence of a single-point blow-up for the solution $u$ of problem \eqref{rnstregm1}-\eqref{rnstregm3}.
\end{remark}
\begin{remark}
Notably, by \eqref{eqn:6.38} we conclude that $\Si_\de \to 0$ as $\de\to 0,$ i.e. the more spiky  initial data we consider then the faster the diffusion-driven blow-up for $z$ and thus  for $u$  take place.
\end{remark}
A diffusion-driven instability (Turing instability) phenomenon, as was first indicated in the seminal paper \cite{t52}, is often followed by pattern formation. A similar situation is observed as a consequence of the driven-diffusion finite-time blow-up provided by Theorem \ref{thm:6.1}, and it is described below. The blow-up rate of the solution $u$ of \eqref{rnstregm1}-\eqref{rnstregm3}
and the blow-up pattern (profile) identifying the formed pattern are given.
\begin{theorem}\label{tbu}
Take $N\ge 3,\;\max\{r, \frac{N}{N-2}\}<p<\frac{N+2}{N-2}$ and $\frac{2}{N}<\frac{p-1}{r}<\gamma.$  Assume that both $\Phi(\si)$ and $\Psi(\si)$ are positive and bounded. Then the blow-up rate of the solution of \eqref{rnstregm1}-\eqref{rnstregm3}  can be characterised as follows
\be
\Vert u(\cdot, \si)\Vert_\infty \ \approx \ (\Si_{\max}-\si)^{-\frac{1}{p-1}}, \quad t\uparrow \Si_{\max},\quad\label{ik2}
\ee
 where $\Si_{\max}$ stands for the blow-up time.
\end{theorem}
\begin{proof}
We first perceive that by virtue of \eqref{eqn:6.8} and in view of the H\"{o}lder's inequality, since $p>r,$ it follows
\bge\label{nkl4}
0<K(\si)=\frac{\Psi(\si)e^{(1+r\gamma-p)\int^\si \Phi(s)\,ds}}{\Big(\displaystyle\avint z^r \Big)^{\gamma}}\leq C_1<+\infty.
\ege
Define now $\Theta$ satisfying the partial differential equation
\[\Theta_{\si}=\Delta \Theta+C_1\Theta^{p},\quad\mbox{in}\quad \Om_0\times (0,\Si_{max})\],
with $\frac{\partial \Theta}{\partial \nu}=0,$ on $\partial\Om_0\times(0,\Si_{max})$, and
$\Theta(x,0)=z_0(x)$, in $\Om_0$,
then via comparison $z\leq \Theta$ in $\Om_0\times (0,\Si_{max}).$ Yet it is known, see \cite[Theorem 44.6]{qs07}, that
$
|\Theta(x,\si)|\leq C_{\eta}|x|^{-\frac{2}{p-1}-\eta}\quad\mbox{for}\quad \eta>0,
$
and thus
\bge\label{tbsd18}
|z(x,\si)|\leq C_{\eta}|x|^{-\frac{2}{p-1}-\eta}\quad\mbox{for}\quad (x,\si)\in \Om_0\times (0, \Si_{max}).
\ege
Following the same steps as in proof of \cite[Theorem 9.1]{KS16}  we derive
\bge\label{tbsd18a}
\lim_{\si\to \Si_{max}} K(\si)=\omega\in(0,+\infty).
\ege
By virtue of \eqref{tbsd18a} and applying  \cite[Theorem 44.3(ii)]{qs07} we can find a constant $C_{U}>0$ such that
\bge\label{ube}
\left|\left|z(\cdot,\si)\right|\right|_{\infty}\leq C_{U}\left(\Si_{max}-\si\right)^{-\frac{1}{(p-1)}}\quad\mbox{in}\quad (0, \Si_{max}).
\ege
Setting $N(\si):=\left|\left|z(\cdot,\si)\right|\right|_{\infty}=z(0,\si),$ then $N(\si)$ is differentiable for almost every $\si\in(0,\Si_{\de}),$ in view of  \cite{fmc85}, and
$
\frac{dN}{d\si}\leq K(\si) N^p(\si).
$
Notably $K(\si)\in C([0,\Si_{\max}))$ and owing to \eqref{nkl4} it is bounded in any time interval $[0,\si],\; \si<\Si_{\max};$ then upon integration we obtain
\bge\label{lbe}
\left|\left|z(\cdot,\si)\right|\right|_{\infty}\geq C_L\left(\Si_{\max}-\si\right)^{-\frac{1}{(p-1)}}\quad\mbox{in}\quad (0, \Si_{\max}),
\ege
for some positive constant $C_L.$

Recalling that  $z(x,\si)=e^{\int^{\si}\Phi(s)\,ds} u(x,\si)$ then \eqref{ube} and \eqref{lbe} entail
\bgee
\widetilde{C}_L\left(\Si_{max}-\si\right)^{-\frac{1}{(p-1)}}\leq\left|\left|u(\cdot,\si)\right|\right|_{\infty}\leq \widetilde{C}_U\left(\Si_{max}-t\right)^{-\frac{1}{(p-1)}}\quad\mbox{for}\quad \si\in(0, \Si_{max}),
\egee
where now $\widetilde{C}_L, \widetilde{C}_U$ depend on $\Si_{max},$
and thus \eqref{ik2} is proved.
\end{proof}

\begin{remark}\label{nkl7}
We first perceive that \eqref{tbsd18} provides a rough form of the blow-up pattern for $z$ and thus for $u$ as well. Additionally, owing to \eqref{nkl4} the non-local problem \eqref{tbsd3}-\eqref{tbsd3b} can be treated as a local one for which the more accurate asymptotic blow-up profile, \cite{mz98}, is known and is given by
$
\lim_{\si\to \Si_{max}}z(|x|,\si)\sim C\left[\frac{|\log |x||}{|x|^2}\right]\;\mbox{for}\; |x|\ll 1,\;\mbox{and}\; C>0.
$
Using again the relation between $z$ and $u$ we end up with a similar asymptotic blow-up profile for the diffusion-driven-induced blow-up solution $u$ of problem \eqref{rnstregm1}-\eqref{rnstregm3}.  This blow-up profile actually determines the form of the developed patterns which are induced as a result of the diffusion-driven instability and it is numerically investigated in the next section.
\end{remark}
\section{Numerical Experiments}\label{num-sec}
To illustrate some of the theoretical results of the previous sections we perform a series of numerical experiments for which we solve the involved PDE problems using the Finite Element Method \cite{j87}, using piecewise linear basis functions and implemented using the adaptive finite-element toolbox ALBERTA \cite{ss05}. In all our simulations (unless stated otherwise) the domain was triangulated using 8192 elements, the discretisation in time was done using the forward Euler method taking $5\times 10^{-4}$ as time-step and the resulting linear system solved using Generalized Minimal Residual iterative solver \cite{sa03}.

\subsection{Experiment 1}
We take an initial condition $u_0(x)$ and a set of parameters satisfying the assumptions of Theorem~\ref{thm1} and solve \eqref{nstregm1t}-\eqref{nstregm3t} on $\Omega_0=\left[0, 1 \right]^2$. The initial condition is chosen
$
u_0(x,0)=\cos(\pi y)+2.
$
As for the domain evolution we consider four different cases:
\begin{itemize}
\item $\rho(t)=e^{\beta t}$ (exponentially growing domain);
\item $\rho(t)=e^{-\beta t}$ (exponentially decaying domain);
\item $\rho(t)=\frac{e^{\beta t}}{1+\frac{1}{m}\left(e^{\beta t}-1\right)}$ (logistically growing domain);
\item $\rho(t)=1$ (static domain).
\end{itemize}

We summarise all parameters used in Table \ref{table1}. In Fig.~\ref{figure1}, we demonstrate the $||u(x,t)||_\infty$ for each of the domain evolutions, so we can monitor their respective blow-up times.
\begin{table}[h!]\label{table1}
\begin{center}
\begin{tabular}{c|c|c|c|c|c|c}
$D_1$ & $p$ & $q$ & $r$ & $s$ & $\beta$ & $m$\\ \hline
1 & 3 & 2 & 1 & 2 & 0.1 & 1.5
\end{tabular}
\caption{Set of parameters used in Experiment 1.}
\end{center}
\end{table}
\begin{figure}[ht!]\label{figure1}
\begin{center}
\includegraphics[trim=0mm 0mm 0mm 0mm, clip,scale=.3]{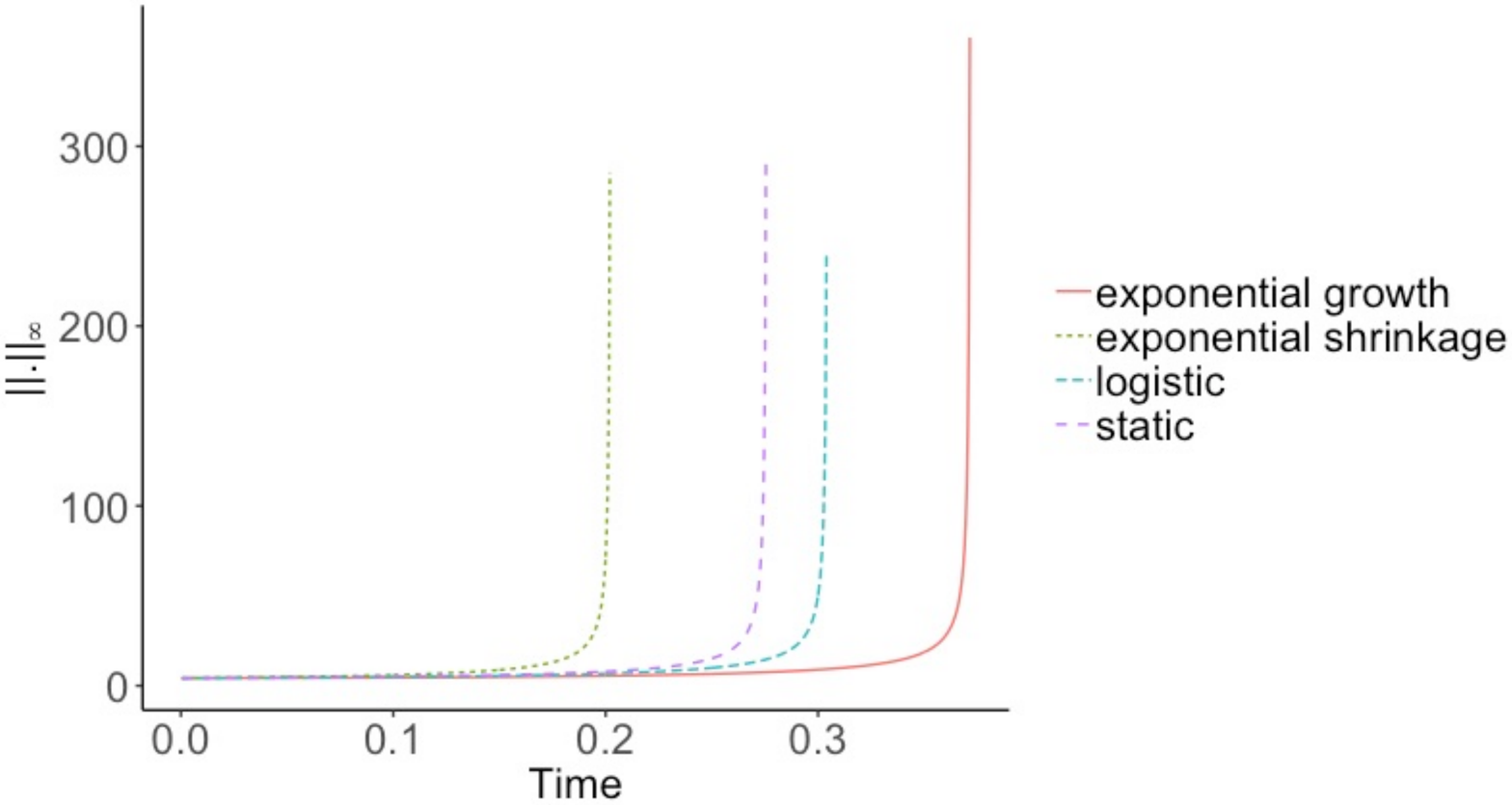}
\caption{Plots representing $||u(x,t)||_\infty$, where $u(x,t)$ is the numerical solution of \eqref{nstregm1t}-\eqref{nstregm3t} for different domain evolutions: static, exponentially decaying and growing, and logistically growing domains, starting from the initial condition $u_0=\cos(\pi y)+2$ in the unit square. Parameters are shown in Table~\ref{table1} and satisfy conditions of Theorem~\ref{thm1}. (Colour version online).}
\end{center}
\end{figure}
If we denote by $\Sigma_g$, $\Sigma_s$, $\Sigma_{lg}$and $\Sigma_1$ the blow-up times for the case of exponentially growing and decaying, the logistically growing domains and the static domain, respectively, we observe from Fig.~\ref{figure1} that we have the following relation
$\Sigma_g>\Sigma_{lg}>\Sigma_1>\Sigma_s,$
which is in agreement with the mathematical intuition.

We now take the same initial condition, $u_0$ and the same initial domain which we assume is evolving exponentially and consider parameters $D_1=1$, $p=1.4$, $q=1$, $r=1$ and $s=2$ for which inequality \eqref{nk11} of Remark~\ref{rem1a} holds. As we can see in Fig.~\ref{figure2}, we have an example of a solution $\bar{u}$ that does not blow up, as already conjectured in the aforementioned remark. Notably, this numerical experiment predicts a very interesting phenomenon both mathematically and biologically. It predicts the infinite-time quenching of the solution of problem \eqref{nstregm1t}-\eqref{nstregm3t} and thus the extinction of the activator in the long run, see also Remark \ref{nny}. Note also that this result is not in contradiction with Proposition \ref{eiq}, where infinite-time quenching is ruled out since condition \eqref{mts2an} is not satisfied for an exponentially growing domain where $\Phi(\si)$ is an unbounded function as indicated in Remark \ref{aal5}.

\begin{figure}[htp!]\label{figure2}
\begin{center}
\includegraphics[trim=0mm 0mm 0mm 0mm, clip,scale=0.2]{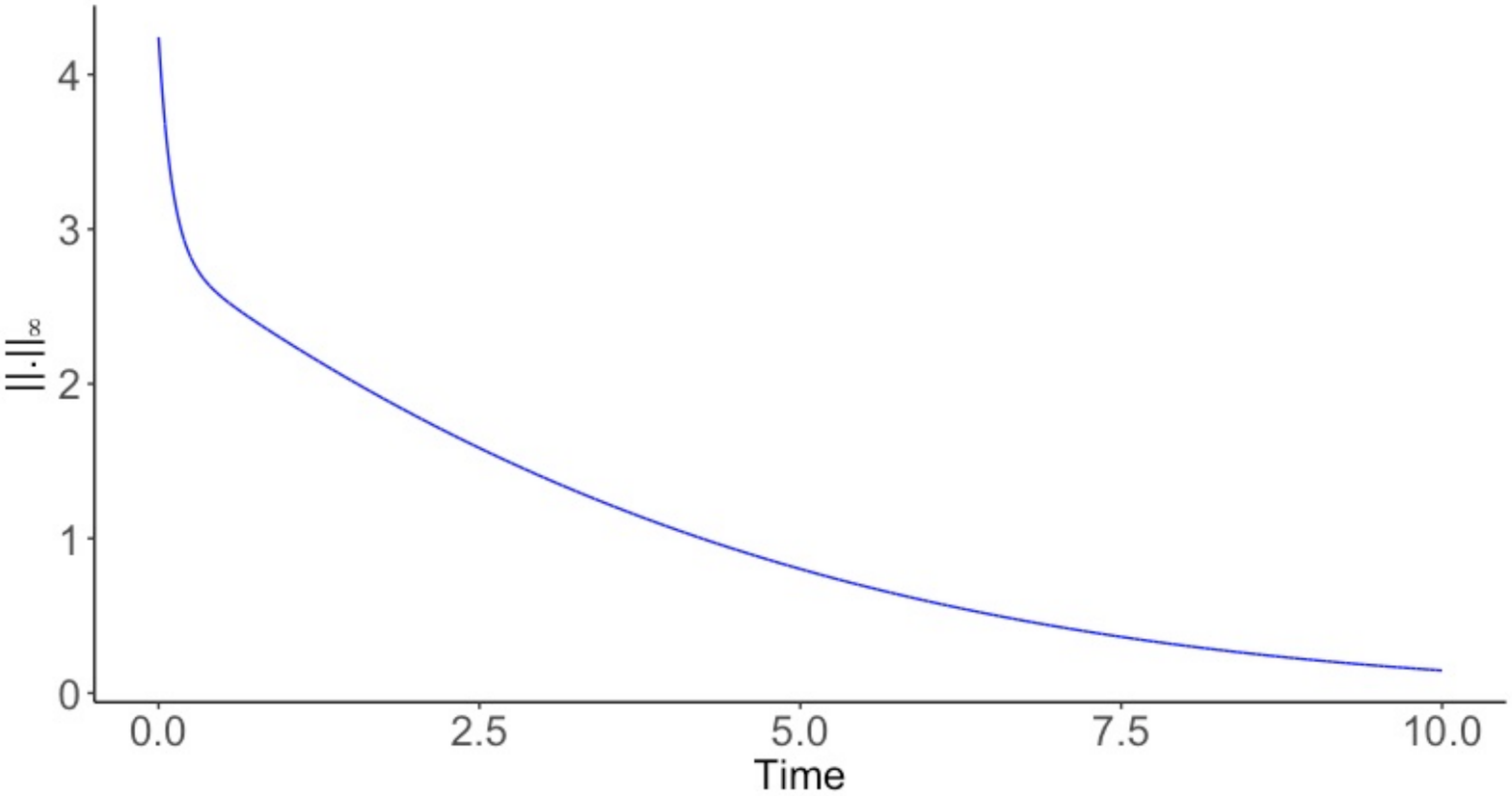}
\caption{The plot of $||\bar{u}||_\infty$ resulting from the numerical solution of \eqref{nstregm1t}-\eqref{nstregm3t} considering the unit square as initial domain, evolving exponentially, considering parameters $D_1=1$, $p=1.4$, $q=1$, $r=1$ and $s=2$. (Colour version online).}
\end{center}
\end{figure}

\subsection{Experiment 2}\label{exp2}

This experiment is meant to illustrate Theorem~\ref{thm4} and we take the same initial data $u_0=\cos(\pi y)+2$ and take $\Omega_0$ as the unit square when numerically solving equations \eqref{nstregm1t}-\eqref{nstregm3t}. As for domain evolution we  consider $\rho(t)=e^{\beta t}$, with $\beta=0.1$. To proceed, we consider two sets of parameters, one for which assumptions of Theorem~\ref{thm4} are satisfied and another for which those assumptions are not fulfilled. See Table~\ref{table2} for model parameters.
\begin{table}[h!]\label{table2}
\begin{center}
\begin{tabular}{c|c|c|c|c|c}
conditions of Th.~\ref{thm4}& $D_1$ & $p$ & $q$ & $r$ & $s$ \\ \hline
are verified & 1 & 1 & 2 & 3 & 2 \\
are not verified & 1 & 3 & 2 & 1 &1
\end{tabular}
\caption{Set of parameters used in Experiment 2.}
\end{center}
\end{table}

\begin{figure}[htp!]\label{figure3}
\begin{center}
\subfigure[Conditions of Th. \ref{thm4} are met ($p=1, q=2, r=3, s=2$).]{
{\includegraphics[trim=0mm 10mm 0mm 10mm, clip,scale=.2]{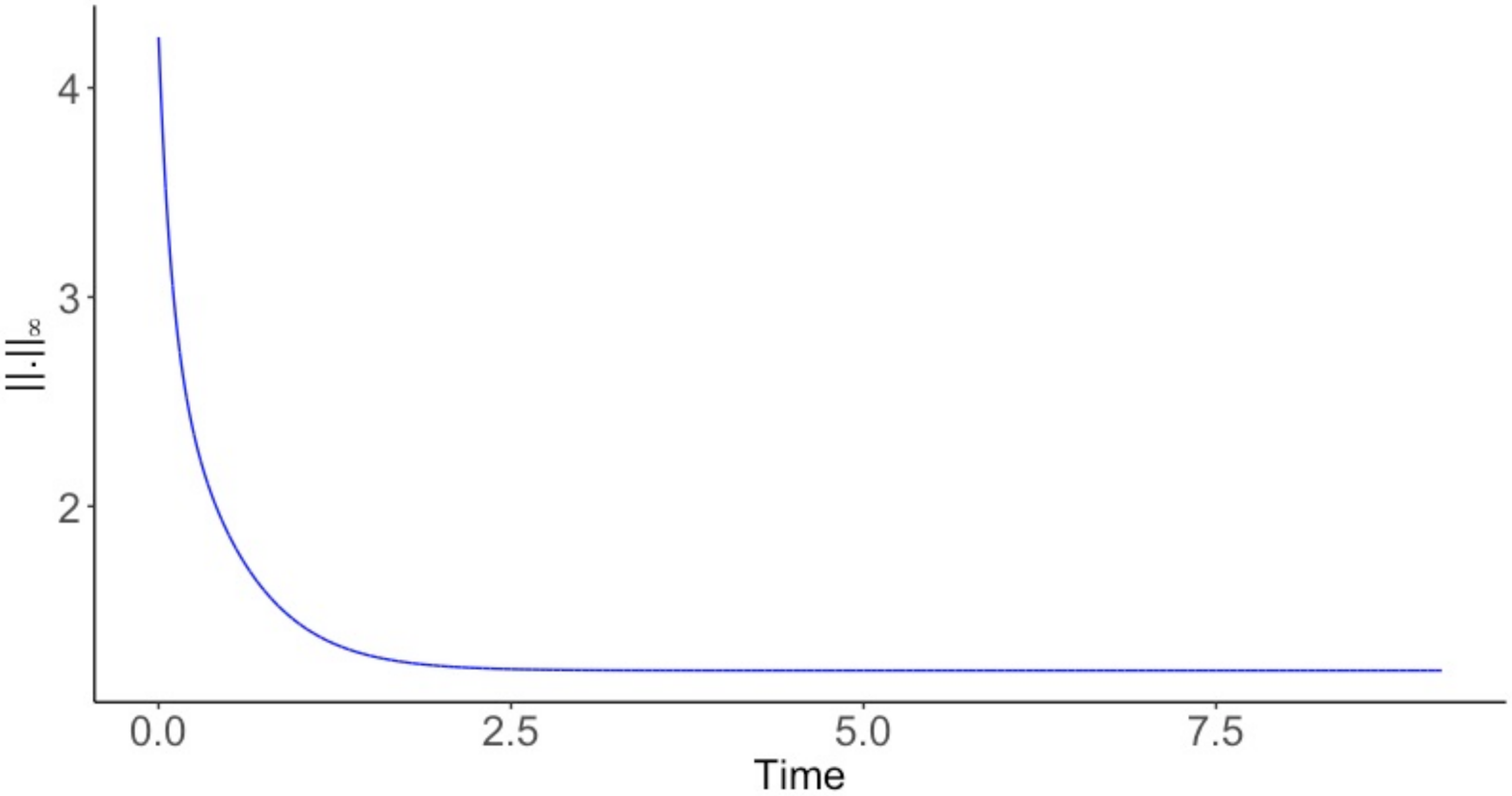}}}
\subfigure[Conditions of Th. \ref{thm4} are not met ($p=3, q=2, r=1, s=2$).]{
{\includegraphics[trim=0mm 10mm 0mm 10mm, clip,scale=.2]{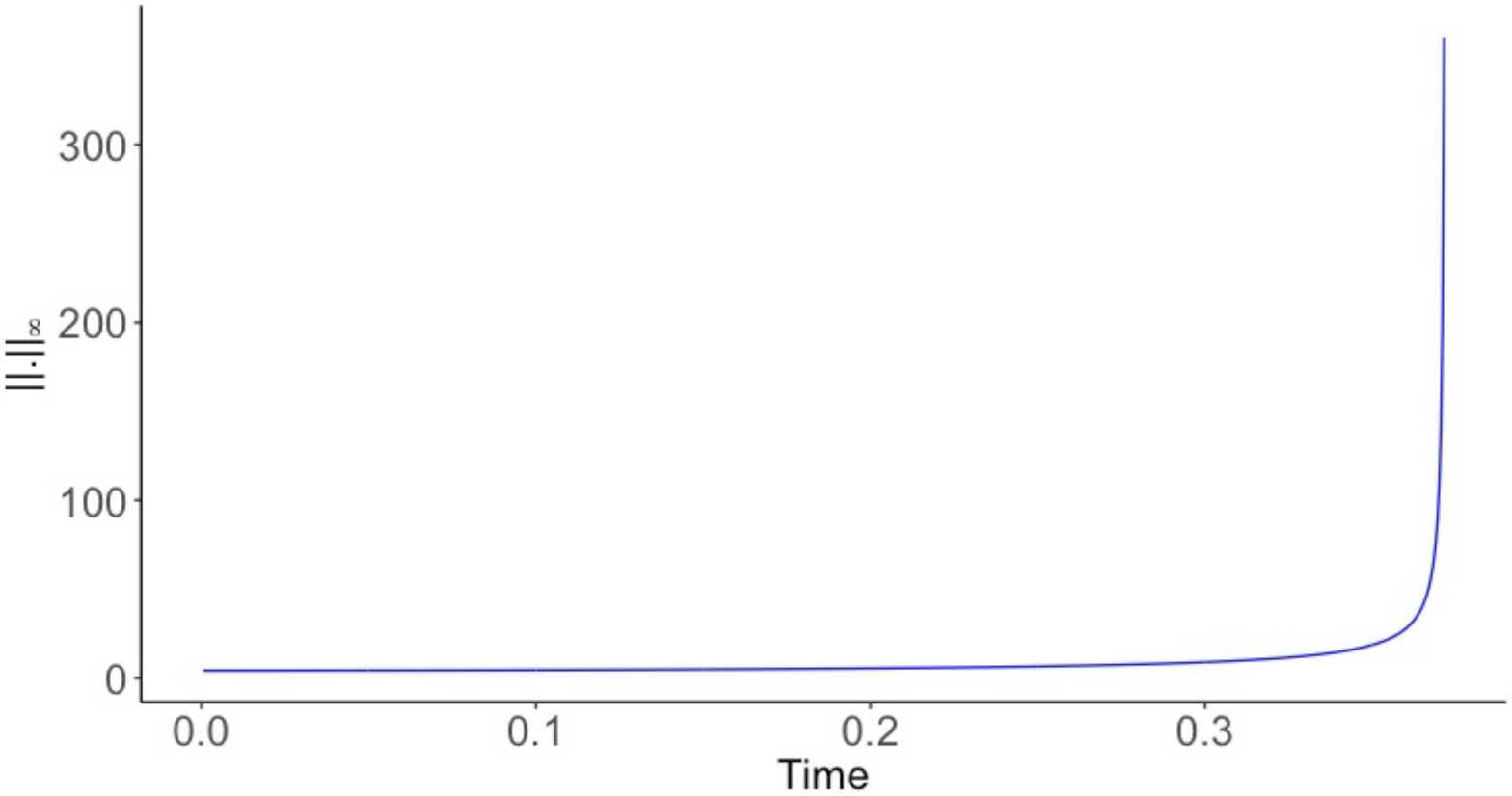}}}
\caption{The plot of $||u(x,t)||_\infty$, where $u(x,t)$ is the numerical solution of \eqref{nstregm1t}-\eqref{nstregm3t}. Initial condition is $u_0=\cos(\pi y)+2$ and $\Omega_0$ is the unit square evolving according to exponential growth ($\beta=0.1$). (Colour version online).}
\end{center}
\end{figure}

Results shown in Fig.~\ref{figure3} are in agreement with theoretical predictions of Theorem~\ref{thm4} since the solutions exists for all times when the assumption of the theorem are met (Fig.~\ref{figure3}(a)), otherwise a finite-time blow-up is exhibited to occur(Fig.~\ref{figure3}(b)).

\subsection{Experiment 3}
In this experiment we intend to illustrate Theorem~\ref{thm5} so we numerically solve \eqref{nstregm1t}-\eqref{nstregm3t} in $\R^3$, taking $\Omega_0$ as the unit sphere and initial condition $u_0$ given by \eqref{eqn:6.1}, considering $\delta=0.8$ and $\lambda=0.1.$  As for other parameters we choose $D_1=1$, $p=4$, $q=4$, $r=2$ and $s=1$, which satisfy the conditions of the theorem. In Fig.~\ref{figure4} we display the $L_\infty-$norm of the solution $u$ for three types of evolution laws implemented, namely: exponential decay, logistic decay and no evolution. For the exponential and logistic decay we select the same set of parameters as used in Experiment 1. As we can observe, for all the cases the solution blows up, as theoretically predicted by Theorem~\ref{thm5}. Again the blow-up times have the order
$
\Sigma_1>\Sigma_{ls}>\Sigma_s,
$
where now $\Sigma_{ls}$ stands for the blow-up time for the logistic decay evolution, beeing in agreement with the mathematical intuition.
\begin{figure}[ht!]\label{figure4}
\begin{center}
\includegraphics[trim=0mm 0mm 0mm 0mm, clip,scale=.25]{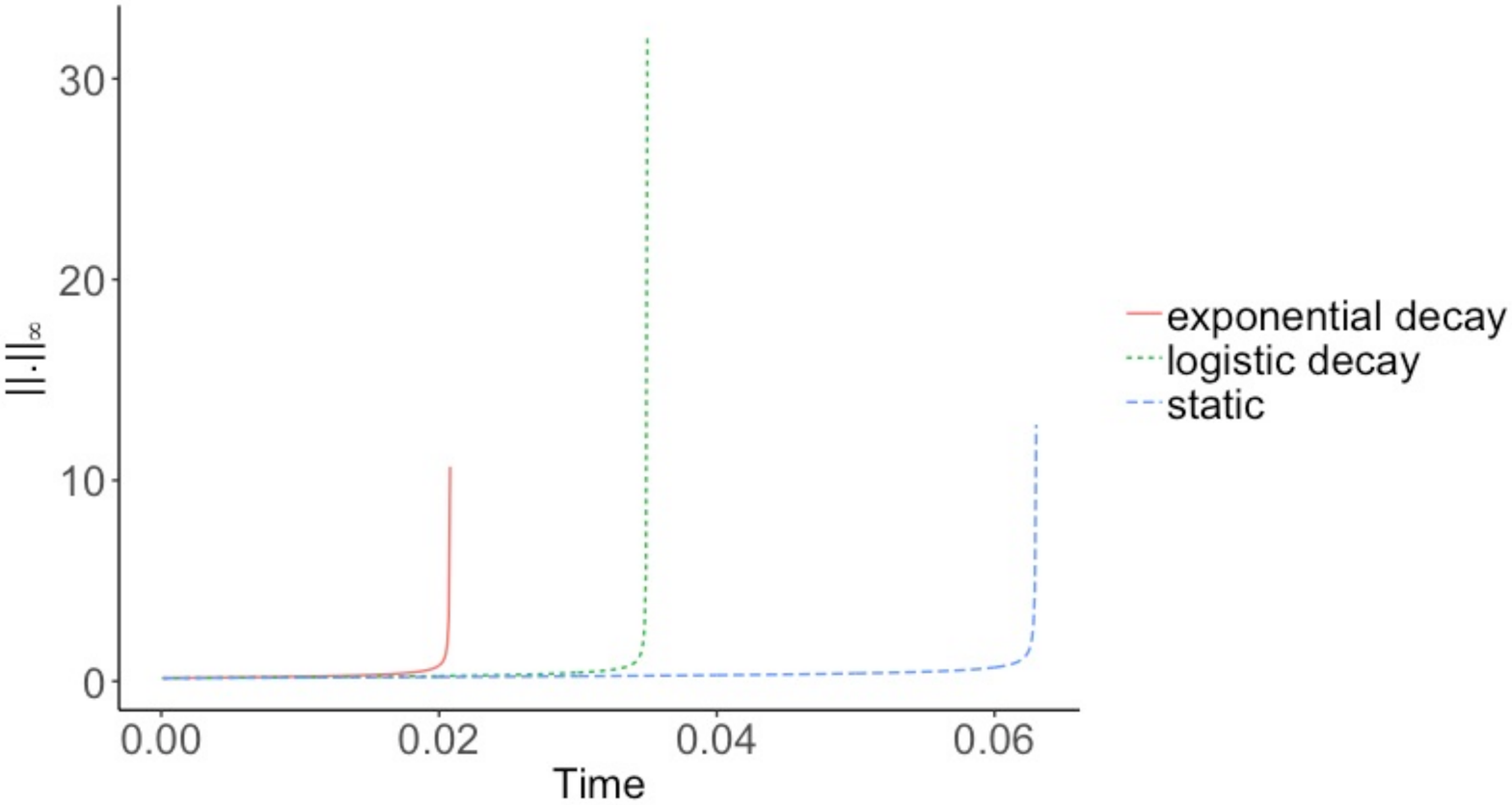}
\caption{Plots for $||u(x,t)||_\infty$, where $u(x,t)$ is the numerical solution of \eqref{nstregm1t}-\eqref{nstregm3t}, in $\R^3$, considering $\Omega_0$ the unit sphere. Three evolution laws considered: exponential decay, logistic decay and no evolution (static domain). Parameters used: $p=4, q=4, r=2, s=1$ and initial condition given by \eqref{eqn:6.1} taking $\delta=0.8$ and $\lambda=0.1$. (Colour version online).}
\end{center}
\end{figure}
In Fig.~\ref{figure5-a} and Fig.~\ref{figure5-b} we compare the initial solution with the solution at $t=0.03$ respectively, for the logistic decay, close to the blow-up time $t=0.03$, by looking at a cross section of the three-dimensional unit sphere $\Omega_0.$ Besides, in Fig.~\ref{figure5-c} and Fig.~\ref{figure5-d} again the solution at section  cross of $\Omega_0$ is depicted but now for the stationary and exponential decaying case respectively.  Through this experiment we can observe the formation of blow-up (Turing-instability) patterns around the origin $R=0.$ We actually conclude that the evolution of the domain has no impact on the form of blow-up patterns, however it definitely affects the spreading of Turing-instability patterns as it is obvious from Fig.~\ref{figure5-b}, Fig.~\ref{figure5-c} and Fig.~\ref{figure5-d}.


\begin{figure}%
\centering
\subfigure[][]{%
\label{figure5-a}%
\includegraphics[height=1.2in]{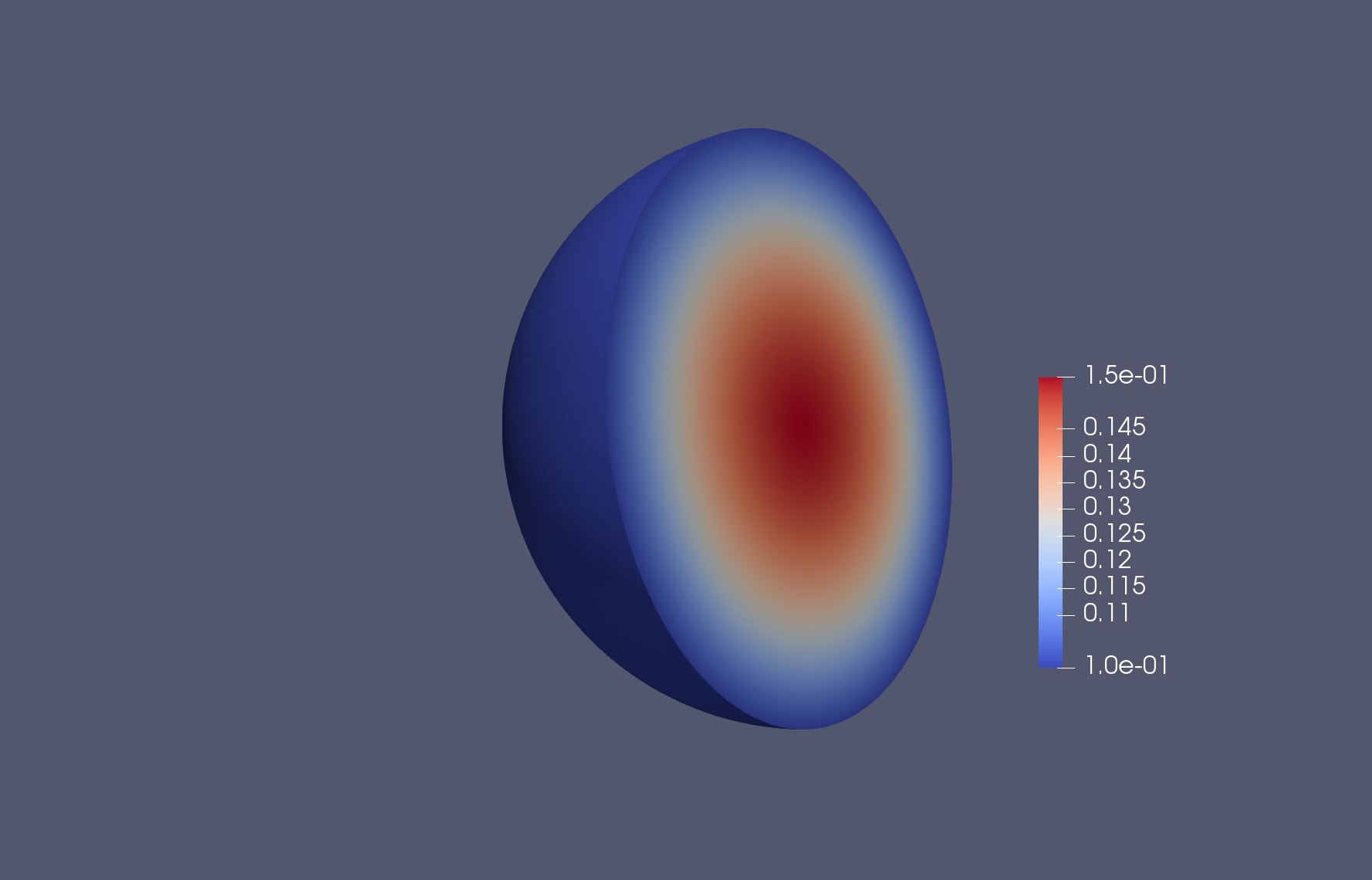}}%
\hspace{5pt}%
\subfigure[][]{%
\label{figure5-b}%
\includegraphics[height=1.2in]{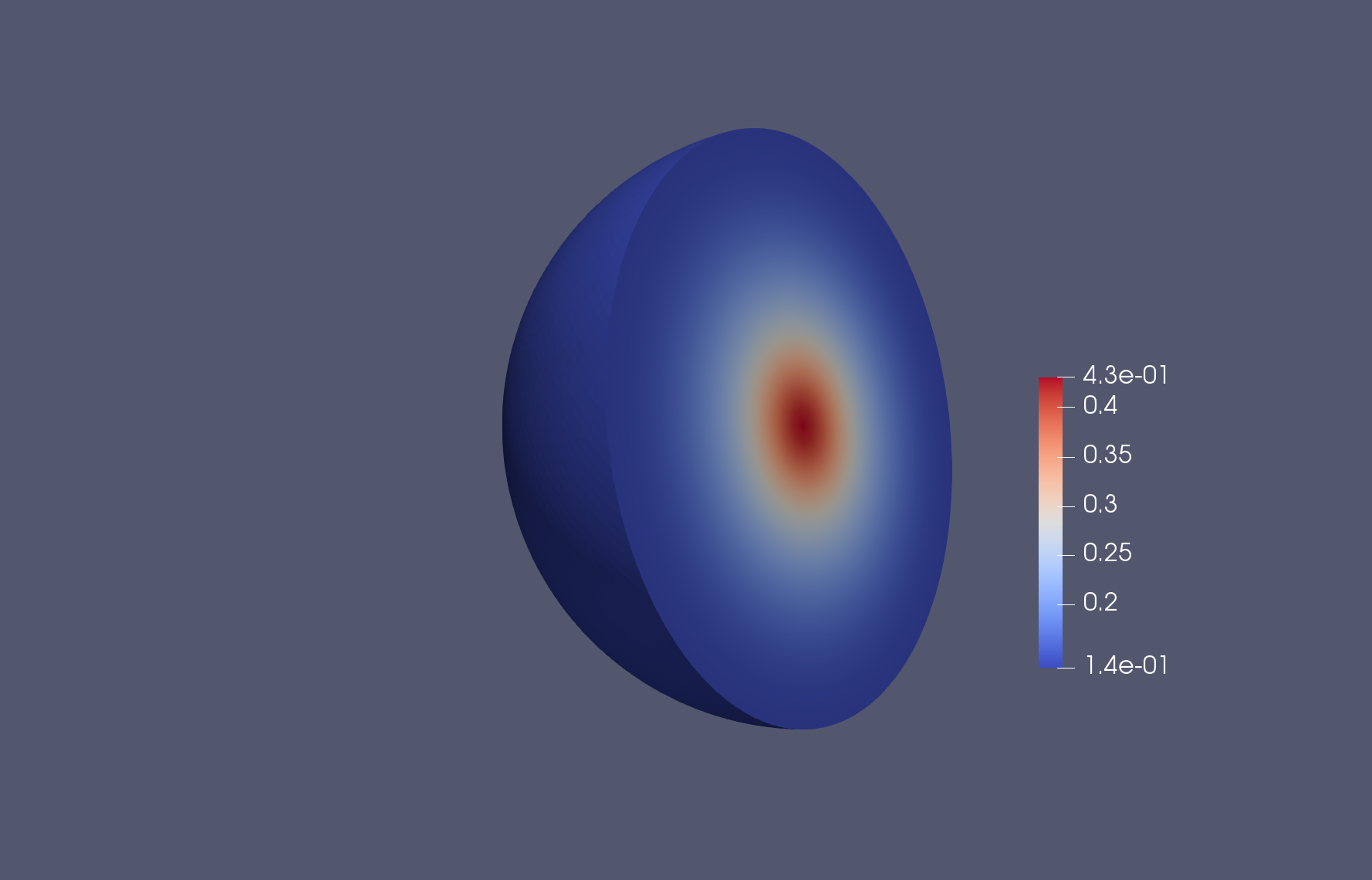}} \\
\subfigure[][]{%
\label{figure5-c}%
\includegraphics[height=1.2in]{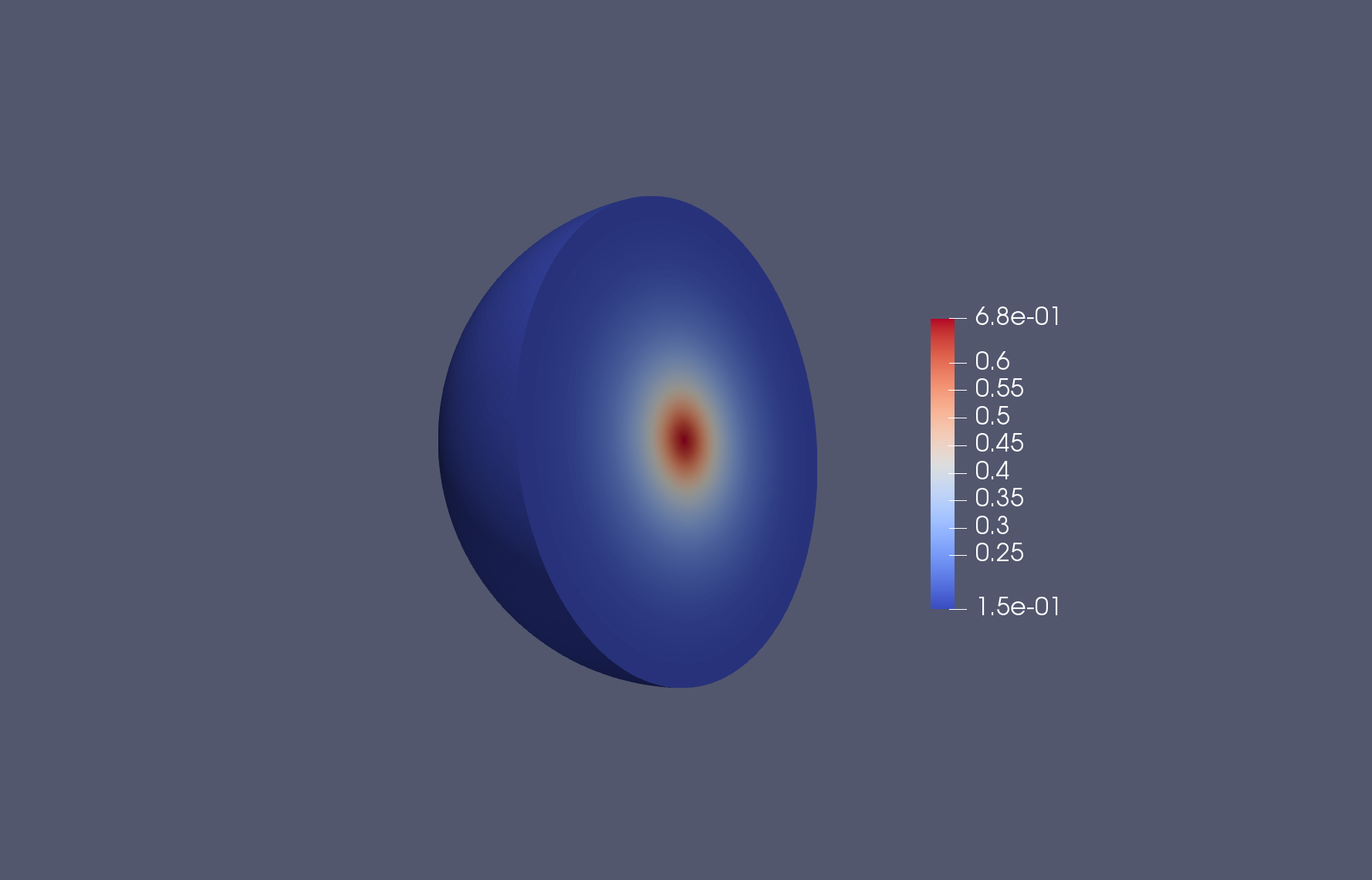}}%
\hspace{5pt}%
\subfigure[][]{%
\label{figure5-d}%
\includegraphics[height=1.2in]{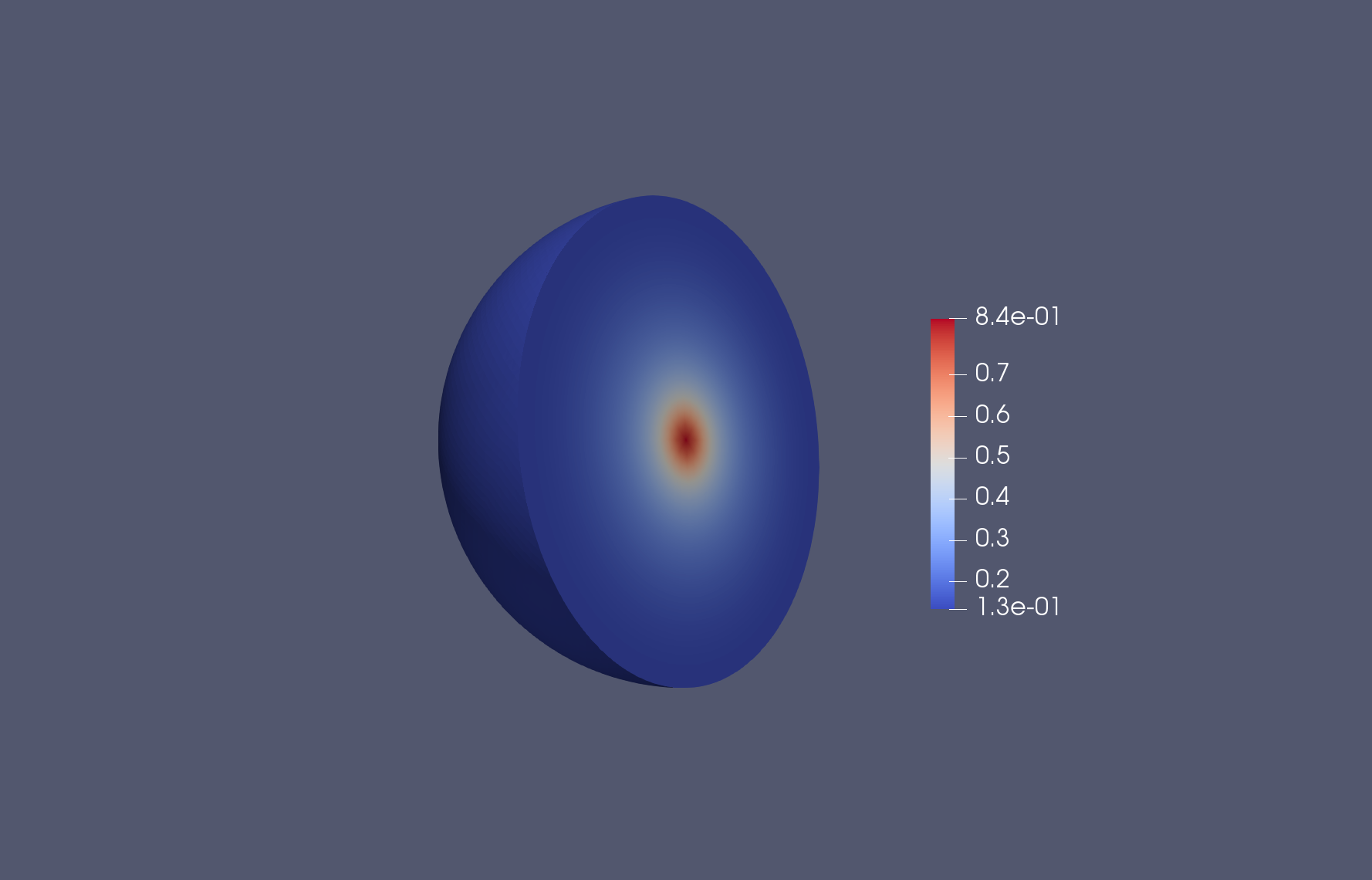}}%
\caption[A set of four subfigures.]{Numerical solution of experiment 3. The above figures actually show the blow-up  (Turing-instability) patterns on a cross-section of the three-dimensional sphere $\Omega_0.$:
\subref{figure5-a} Initial profile of solution for logistic growth;
\subref{figure5-b} blow-up pattern for logistic growth at $t=0.03$;
\subref{figure5-c} blow-up pattern for  stationary case at $t=0.07$; and,
\subref{figure5-d} blow-up pattern for exponential decay at $t=0.03$.(Colour version online)}%
\label{figure5}%
\end{figure}

\subsection{Experiment 4}
Next we design a numerical experiment to compare the dynamics of the reaction-diffusion system \eqref{tregm1}-\eqref{tregm2} with that of the non-local problem \eqref{nstregm1t}-\eqref{nstregm3t} under the assumptions of Theorem~\ref{thm1}. To this end we perform an experiment considering $u_0=\hat{u}_0=\cos(\pi y)+2$, $\Omega_0=\left[0, 1\right]^2$, $p=3$, $q=2$, $r=1$ and $s=2$. For the reaction-diffusion system \eqref{tregm1}-\eqref{tregm2} we also take in addition $D_1=0.01$, $D_2=1$, $\tau=0.01$ and $v_0=2$ whilst for \eqref{nstregm1t}-\eqref{nstregm3t} we only choose $D_1=0.01.$ For both cases we consider an exponential decaying evolution, with $\beta=0.1$. Unlike previous numerical examples, here the domain was triangulated using $786432$ elements with timestep $10^{-4}$.

The obtained results are displayed in Fig.~\ref{figure6} and they actually demonstrate that reaction-diffusion system \eqref{tregm1}-\eqref{tregm2} and  non-local problem \eqref{nstregm1t}-\eqref{nstregm3t} share the same dynamics. In particular the solutions of both problems exhibit  blow-up which takes place in finite time. The latter, biologically speaking, means that in the examined case we just need to monitor only the dynamics of the activator, whose dynamics governed by non-local problem \eqref{nstregm1t}-\eqref{nstregm3t}. Then we can get an insight regarding the interaction between both of the chemical reactants (activator and inhibitor) provided by  reaction-diffusion system \eqref{tregm1}-\eqref{tregm2}.

\begin{figure}[htp!]\label{figure6}
\begin{center}
\subfigure[$||\hat{u}||_\infty$ for the numerical solution of \eqref{tregm1}-\eqref{tregm2}.]{
{\includegraphics[trim=0mm 0mm 0mm 0mm, clip,scale=.15]{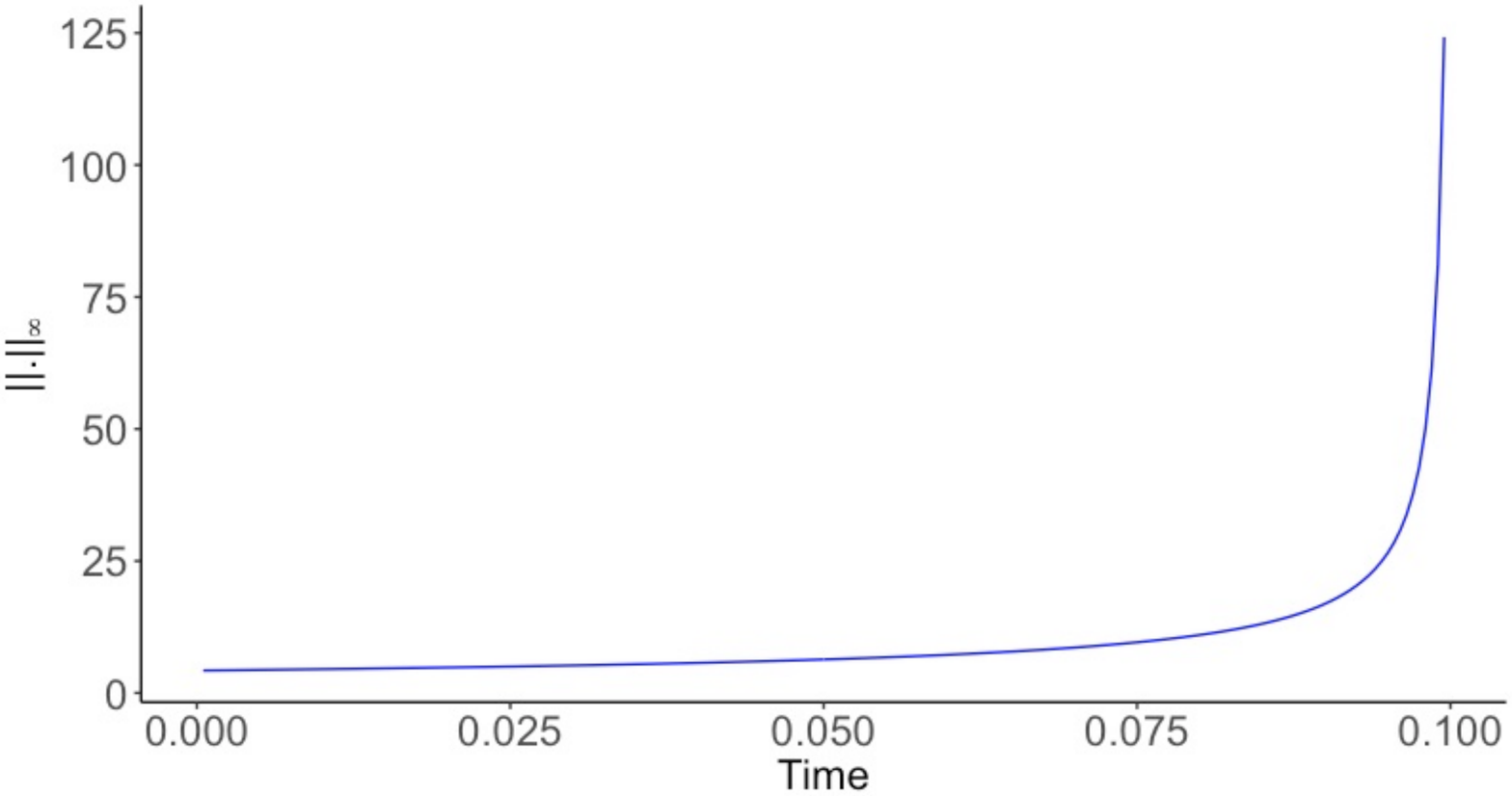}}}
\subfigure[$||u||_\infty$ for the numerical solution of \eqref{nstregm1t}-\eqref{nstregm3t}.]{
{\includegraphics[trim=0mm 0mm 0mm 0mm, clip,scale=.15]{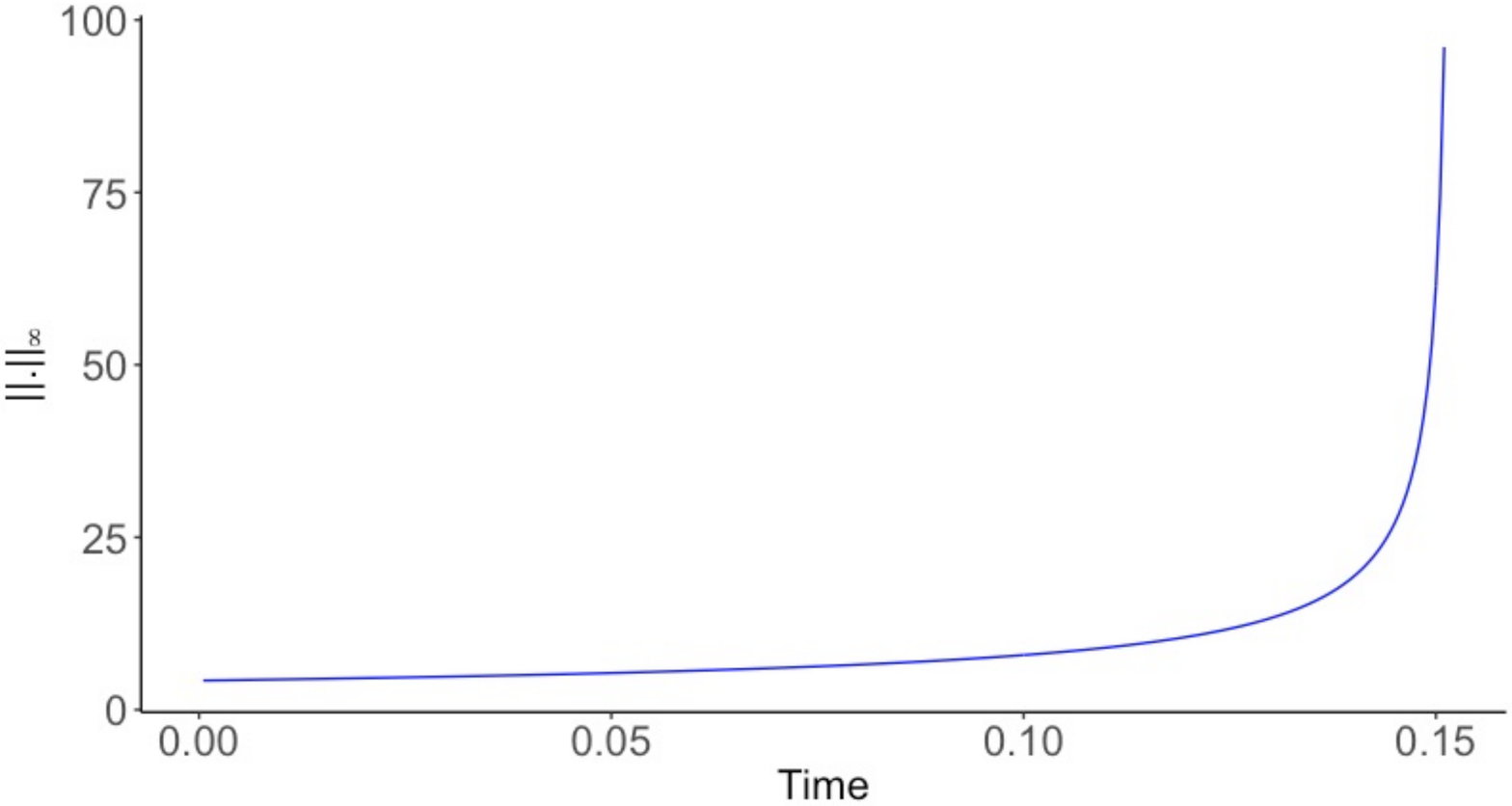}}}
\caption{Plots of the $L_\infty$ norm for the numerical solution of \eqref{tregm1}-\eqref{tregm2} and \eqref{nstregm1t}-\eqref{nstregm3t}. Initial condition is $u_0=\hat{u}_0=\cos(\pi y)+2$ and $\Omega_0$ is the unit square, exponentially decaying ($\beta=0.1$). (Colour version online).}
\end{center}
\end{figure}



\begin{thebibliography}{99}

\bibitem{bk18} A. Bobrowski \& M. Kunze, {\it Irregular convergence of mild solutions of semilinear equations}, {\bf arXiv:1807.04815v1.}

\bibitem{br11} H. Brezis, {\it Functional Analysis, Sobolev Spaces and Partial Diffrential Equations}, Springer 2011.

\bibitem{ev10} L.C. Evans, {\it Partial Differential Equations}, Graduate Studies in Mathematics Vol. 19, 2nd Edition,  AMS 2010

\bibitem{fmc85}A. Friedman \& J.B. McLeod, {\it Blow-up of positive solutions of semilinear heat equations}, Indiana Univ. Math. J. {\bf 34} (1985) 425-447.

\bibitem{gm72}A. Gierer \& H. Meinhardt, {\it A theory of biological pattern formation}, Kybernetik (Berlin) {\bf 12} (1972) 30-39.

\bibitem{hy95}B. Hu \& H-M. Yin, {\it Semilinear parabolic equations with prescribed energy}, Rend. Circ. Mat. Palermo {\bf 44} (1995) 479-505.



\bibitem{KS16} N.I. Kavallaris \& T. Suzuki, {\it On the dynamics of a non-local parabolic
equation arising from the GiererâMeinhardt
system }, Nonlinearity {\bf 30} (2017) 1734â-1761.

\bibitem{KS18} N.I. Kavallaris \& T. Suzuki, {\it  Non-Local Partial Differential Equations for Engineering and Biology:
Mathematical Modeling and Analysis}, Mathematics for Industry Vol. 31 Springer Nature 2018.

\bibitem{ke78}J. Keener, {\it Activators and inhibitors in pattern formation}, Stud. Appl. Math. {\bf 59} (1978) 1--23.


\bibitem{L11} M. Labadie, {\it The stabilizing effect of growth on pattern formation}, preprint.

\bibitem{Lie96} G.M. Lieberman,  {\it Second order parabolic differential equations}, World Scientific Publishing Co., Inc., River Edge, NJ, 1996.




\bibitem{mm07} A. Madzamuse \& P.K. Maini, {\it Velocity-induced numerical solutions of reaction-diffusion systems on continuously growing domains} J. Comput. Physics, {\bf 225} (2007) 100--119.



\bibitem{mz98} F. Merle \& H. Zaag, {\it Refined uniform estimates at blow-up and applications for nonlinear heat equations}, Geom. Funct. Anal. {\bf 8(6)} (1998), 1043--1085.

\bibitem{nst06}W.-M. Ni, K. Suzuki \& I. Takagi, {\it The dynamics of a kinetic activator-inhibitor system}, J. Differential Equations, {\bf 229}(2006) 426--465.

\bibitem{n11} W.-M. Ni, \textit{The Mathematics of Diffusion}  CBMS-NSF Series, SIAM 2011.

\bibitem{pspbm04} R. G. Plaza, F. Sanchez-Garduno, P. Padilla, R.A. Barrio \& P.K.  Maini, {\it The effect of growth and curvature on pattern formation}, J. Dynam. Differential Equations {\bf 16(4)} (2004),  1093--1121.

\bibitem{qs07}P. Quittner \& Ph. Souplet, {\it Superlinear parabolic problems. Blow-up, global existence and steady states}, Birkh\"auser Verlag, Basel, 2007.



\bibitem{j87} C. Johnson {\it Numerical Solution of Partial Differential Equations by the Finite Element Method}, Cambridge University Press. (1987)


\bibitem{ss05} A. Schmidt \& K. G. Siebert, {\it Design of adaptive finite element software: the finite element toolbox ALBERTA}, Springer. (2005)

\bibitem{sa03} Y. Saad, {\it Iterative Methods for Sparse Linear Systems}, 2nd Edition, SIAM 2003.

\bibitem{t52} A.M. Turing, {\it The chemical basis of morphogenesis}, Phil. Trans. Roy. Soc. B {\bf 237} (1952), 37--72.







\end{thebibliography}
\end{document}